\documentclass[11pt,a4paper]{article}

\usepackage[latin1]{inputenc}
\usepackage[ngerman,english]{babel}

\usepackage{amsmath,amssymb}
\usepackage{
						amsthm
           ,hyperref
           ,bbm
           ,enumerate           }

\usepackage[round,authoryear]{natbib}

\usepackage[usenames,dvipsnames]{color}
\usepackage[left=3cm,top=3.5cm,right=2cm,bottom=3cm]{geometry}
\newtheorem{theorem}{Theorem}[section]
\newtheorem{proposition}[theorem]{Proposition}
\newtheorem{corollary}[theorem]{Corollary}
\newtheorem{definition}[theorem]{Definition}
\newtheorem{lemma}[theorem]{Lemma}
\newtheorem{remark}[theorem]{Remark}

\definecolor{darkgreen}{rgb}{0.00,0.40,0.00}

\newcommand{\1}{\mathbbm{1}}

\newcommand{\IE}{\mathbb{E}}
\newcommand{\E}{\mathbb{E}}
\newcommand{\IF}{\mathbb{F}}
\newcommand{\IL}{\mathbb{L}}
\newcommand{\IN}{\mathbb{N}}

\newcommand{\IP}{\mathbb{P}}

\newcommand{\IR}{\mathbb{R}}
\newcommand{\R}{\mathbb{R}}

\newcommand{\cF}{\mathcal{F}}

\newcommand{\cH}{\mathcal{H}}

\newcommand{\cS}{\mathcal{S}}

\newcommand{\ud}{\mathrm{d}}
\newcommand{\uds}{\mathrm{d}s}
\newcommand{\udt}{\mathrm{d}t}

\newcommand{\udws}{\mathrm{d}W_s}

\newcommand{\ti}{{t_i}}
\newcommand{\tj}{{t_j}}
\newcommand{\tip}{{t_{i+1}}}

\newcommand{\ax}{{\alpha_{\scriptscriptstyle{\mathcal{X}}}}}
\newcommand{\ay}{{\alpha_{\scriptscriptstyle{\mathcal{Y}}}}}
\newcommand{\az}{{\alpha_{\scriptscriptstyle{\mathcal{Z}}}}}

\title{FBSDE with time delayed generators: Lp-solutions, differentiability, representation formulas and path regularity\footnote{Gon\c calo dos Reis kindly acknowledges financial support by the \emph{DFG Research Center MATHEON} to visit the Humboldt-Universit\"at zu Berlin as well of the Chair \emph{Financial risks} of the Risk Foundation sponsored by \emph{Soci\'et\'e G\'en\'erale}. Anthony R\'eveillac is grateful to \emph{DFG Research Center MATHEON, project E2} for financial support. Jianing Zhang acknowledges financial support by the \emph{DFG IRTG 1339 SMCP}.}}
\author{	
	\small Gon\c calo dos Reis \\
        \footnotesize  Technische Universit\"at Berlin\\
        \footnotesize  Institut f\"ur Mathematik \\
        \footnotesize  Str. 17 Juni 136 \\
        \footnotesize  10623 Berlin \\
        \footnotesize  dosreis@math.tu-berlin.de
        	\and
	\small Anthony R\'eveillac \\
        \footnotesize  Institut f\"ur Mathematik  \\
        \footnotesize  Humboldt-Universit\"at zu Berlin \\
        \footnotesize  Unter den Linden 6\\
        \footnotesize  10099 Berlin \\
        \footnotesize  areveill@math.hu-berlin.de
        	\and
	\small Jianing Zhang \\
        \footnotesize  Weierstrass Institute for \\
        \footnotesize  Applied Analysis and Stochastics \\
        \footnotesize  Mohrenstra\ss e 39\\
        \footnotesize  10117 Berlin \\
        \footnotesize  jianing.zhang@wias-berlin.de
\vspace*{0.4cm}
}

\begin{document}

\def\linenumberfont{\normalfont\small\sffamily}
\selectlanguage{english}
\maketitle

\begin{abstract}
\noindent 
We extend the work of \cite{DelongImkeller,DelongImkeller2} concerning Backward stochastic differential equations with time delayed generators (delay BSDE). We give moment and a priori estimates in general $L^p$-spaces and provide sufficient conditions for the solution of a delay BSDE to exist in $L^p$. We introduce decoupled systems of SDE and delay BSDE (delay FBSDE) and give sufficient conditions for their variational differentiability. We connect these variational derivatives to the Malliavin derivatives of delay FBSDE via the usual representation formulas. We conclude with several path regularity results, in particular we extend the classic $L^2$-path regularity to delay FBSDE.
\end{abstract}

\medskip

\medskip

\section*{Introduction}

The theory of nonlinear \emph{backward stochastic differential equations} (BSDEs) was introduced by \cite{PardouxPeng90} with its main motivations being mathematical finance (see \cite{97KPQ}) and stochastic control theory (see \cite{YongZhou}). In the last twenty years much effort has been given to this type of equations and nowadays many classes of BSDEs and results on them are available. 
Due to tractability, common results are achieved within a Markovian framework. Under certain conditions the BSDE's solution exhibits a Markov structure and hence can be interpreted as an instantaneous transformation of the underlying Markov process that spans the stochastic basis of the underlying probability space. This in turn yields access to the theory of partial differential equations via the non-linear Feynman-Kac formula.

Moving away from the Markovian setting, \cite{DelongImkeller,DelongImkeller2} introduce a new class of BSDE labeled \textit{backward stochastic differential equations with time delayed generators} (delay BSDEs). The dynamics of these BSDEs are governed by
\begin{align*}
Y_t=\xi + \int_t^T f(s,Y(s),Z(s))\uds - \int_t^T Z_s \udws,\quad t\in[0,T],
\end{align*}
where the generator $f$ at time $s\in[0,T]$ is allowed to depend on the past values of the solution $(Y,Z)$ over the time interval $[0,s]$ and $\xi$ is a measurable random variable. In these two works the authors answered thoroughly several fundamental questions: existence and uniqueness of a square integrable solution, comparison principles, existence of a measure solution, BMO martingale properties for the control component $Z$ of the solution, Malliavin differentiability for delay BSDEs driven by a Wiener process and a generalized Poisson martingale. To the best of our knowledge the only existence and uniqueness results for this class of BSDEs follow from those two works. As pointed out by \cite{Delong2010}, delay BSDEs appear naturally in finance and insurance related problems of pricing and hedging of contracts. In the same work the author analyses a vast scope of contracts to which this class of BSDEs can be applied to.\vspace{0.15cm}

Paying consideration to and seeking reference from the state of the art of BSDEs with non-time delayed generators, the next step concerning delay BSDEs is to obtain a feasible numerical scheme. Here, the main obstacle is the presence of the control process $Z$ in the generator. This process is usually obtained via the predictable representation property of the underlying stochastic basis, and initially all one knows about $Z$ is that it is a square integrable process. To steer in the direction of a numerical scheme a deeper analysis on the fine properties of the solution of such equations is required.  As for numerics for Lipschitz continuous BSDEs (see for example \cite{04BT} or \cite{BenderZhang2008}) one is usually forced to gather several results concerning the \emph{path regularity} properties of the solution process before being able to give proper convergence results. Such path properties include not only sample path continuity but also estimations on the time increments of the components of the solution by the size of the time increment. For the purpose of establishing such path properties we first need to prove several auxiliary results.\vspace{0.15cm}

Our agenda consists of refining and extending the existence and uniqueness results obtained in \cite{DelongImkeller,DelongImkeller2} and then steer into the direction of the smoothness properties of the solution of delay BSDEs. We start by improving the original results of \cite{DelongImkeller} concerning their a priori estimates by reformulating them in a more standard fashion.  
In Lemma 2.1 from \cite{DelongImkeller}, the a priori estimates expresses the difference (in norm) of the solution of two delay BSDE as the difference of the respective terminal conditions and generators. These a priori estimates fall short of the usual a priori estimates one expects to see due to the presence of the solutions of \emph{both} delay BSDE on the right hand side of the estimate. We establish a priori estimates in the classical form where the right hand side of the estimate contains the difference of generators evaluated at their zero spatial state and hence is independent of the BSDE solutions. Within the topic of a priori estimates we extend the results of \cite{DelongImkeller} in another direction. We show that given extra integrability of the terminal condition and the generator, the solution will inherit this integrability. This allows us to state moment and a priori estimates in general $L^p$-spaces and not solely in $L^2$. The proof of these estimates relies on techniques from \cite{DelongImkeller} and on computations carried out for non-time delayed BSDEs in the spirit of \cite{WangRanChen}. The usual techniques to obtain higher order moment estimates fail in the setting of delay BSDEs, the reason for this will be seen in \eqref{eq:auxi1} below. A rough explanation would be that for the usual (non-delay) BSDE setting the dynamics of $Y_t$ is given by sums of Lebesgue and It\^o integrals over the interval $[t,T]$ but for delay BSDEs the dynamics of $Y_t$ depends also on a integral over the whole interval $[0,T]$ which doesn't allow the usual techniques to be used. The general  estimates we obtain pave the way to a result of existence and uniqueness of solutions to delay BSDE with Lipschitz continuous generators in general $L^p$ spaces for $p\geq 2$. Inevitably, in analogy to \cite{DelongImkeller,DelongImkeller2} a compatibility condition on the Lipschitz constant and terminal time is required to obtain existence of solutions (see our Theorem \ref{theo:picard}).\vspace{0.15cm}

A customary field of application of BSDEs consists in coupling them with SDEs, giving rise (in our case) to systems of delay forward-backward SDEs (delay FBSDEs). We show that when coupling a delay BSDE with a forward diffusion and assuming appropriate regularity conditions, we obtain smoothness properties of the solution in terms of the involved parameters, in particular with respect to the initial condition of the forward diffusion. Combining this with the Malliavin differentiability proved in \cite{DelongImkeller2} enables us to derive the usual representation formulas for FBSDE which display the relationship between the Malliavin derivatives of the solution process and their variational (classical) derivatives. It is somewhat surprising that such a relationship still holds since it is usually a consequence of the BSDE's Markov property which clearly fails to materialize in the context of delay FBSDE. \vspace{0.15cm}

With this collection of results we are finally able to address the path regularity issue of delay BSDE. Using the techniques employed in \cite{ImkellerDosReis,pathregcorrection2010}, we establish path continuity for the components of the solution of delay FBSDE and we give a result that bounds the norm of the increments in time of $Y$ and $Z$ by the size of the time increment. We expect that these results will open the door to the derivation of concrete numerical schemes and their convergence rate and intend to tackle these problems in our future research.\vspace{0.15cm}

The paper is organized as follows: in Section \ref{section:prelim} we fix notations and elaborate on the type of time-delayed BSDEs that we consider. In Section \ref{section:apriori} we refine and extend the a priori estimates obtained in \cite{DelongImkeller} and then use them to establish existence and uniqueness of solutions in general $L^p$ spaces. In Section \ref{section:diff} we introduce the delay FBSDE framework and use results from the previous sections to obtain the differentiability of the solution process with respect to the initial state of a forward diffusion. The representation formulas and the path regularity results are presented in Section \ref{section:representation}. 

\section{Preliminaries}
\label{section:prelim}
Let $(\Omega,\cF,\IP)$ be a probability space equipped with a standard $d$-dimensional Brownian motion $W$. For a fixed real number $T>0$ we consider the filtration $\IF:= (\cF_t)_{t\geq 0}$ generated by $W$ and augmented by all $\IP$-null sets. The filtered probability space $(\Omega,\cF,\IF,\IP)$ satisfies the usual conditions. Depending on whether we work on $\IR^d$ or $\IR^{m\times d}$, the Euclidean norm respectively the Hilbert-Schmidt operator norm is denoted by $|\cdot|$. Furthermore, $\nabla$ denotes the canonical gradient differential operator and for a function $h(x,y):\IR^m\times \IR^d\to \IR^n$, we write $\nabla_x h$ or $\nabla_y h$ for the derivatives with respect to $x$ and $y$. We work with the following topological vector spaces:

\begin{itemize}
\item For $p\geq 2$, let $L^p(\IR^m)$ be the space of $\cF_T$-measurable random variables $\xi:\Omega\to\IR^m$ normed by $\|\xi \|_{L^p}:=\IE \big[\,|\xi|^p\, \big]^{1/p}$.
\item 
For $\beta \geq 0$ and $p\geq 1$, $\cH^{p}_\beta(\IR^{m\times d})$ denotes the space of all predictable process $\varphi$ with values in $\IR^{m\times d}$ 
such that the norm $\| \varphi \|_{\cH^p_\beta} := \IE \Big[ \Big(\int_0^T  e^{\beta s} |\varphi_s|^2 \uds \Big)^{p/2} \Big]^{1/p}<\infty$.
\item For $\beta\geq 0$ and $p\geq 2$, $\cS^{p}_\beta(\IR^{m\times d})$ denotes the space of all predictable processes $\eta$ with values in $\IR^{m\times d}$ 
such that the norm $\| \eta \|_{\cS^p_\beta} := \IE \Big[ \Big( \sup_{0 \leq t \leq T}  e^{\beta t} |\eta_t|^2 \Big)^{p/2} \Big]^{1/p}<\infty$.
\end{itemize}
We omit referencing the range space if no ambiguity arises. It is fairly easy to see that for any $\beta, \bar{\beta}\geq 0$ the norms on $\cH_\beta^{p}$, $\cH_{\bar\beta}^{p}$ and $\cS_\beta^{p}$, $\cS_{\bar\beta}^{p}$ are equivalent. 
 	
\subsubsection*{Some notation}
We introduce a notational convention which will be used throughout the text: for an arbitrarily given integrable function $f: [0,T] \to \R^m$, trivially extended to $[-T,0)$ via $f(t) \1_{[-T,0)}(t) = 0$, and a given deterministic finite measure $\alpha$ supported on $[-T,0)$ which is not necessarily atomless, we denote for $t\in [0,T]$ and any $p\geq 2$
\begin{align*}
(f \cdot \alpha)(t) := \int_{-T}^0 f(t+v) \alpha(\ud v)
\quad \text{ and }\quad
(f^p \cdot \alpha)(t) := \int_{-T}^0 |f(t+v)|^p \alpha(\ud v).
\end{align*}
Similarly, for a given process $(\varphi_t)_{t\in[0,T]}$, extended to $[-T,0)$ by imposing $\varphi_t=0$ on $[-T,0)$, we denote 
\begin{equation}
\label{eq:notation1}
(\varphi \cdot \alpha)(t) := \int_{-T}^0 \varphi_{t+v} \alpha(\ud v),\qquad t\in[0,T],
\end{equation}
and
\begin{equation}
\label{eq:notation2}
(\varphi^p \cdot \alpha)(t) := \int_{-T}^0 |\varphi_{t+v}|^p \alpha(\ud v), \quad \quad t\in[0,T], \ p\geq 2.
\end{equation} 
We now give a lemma concerning the change of integration order for \eqref{eq:notation1} and \eqref{eq:notation2}, which will become useful in the sequel.

\begin{lemma}\label{lemma:interchange}
Let $\varphi$ be a process and $\alpha$ a non-random finite measure supported on $[-T,0)$. Then we have the following change of integration order: for every $k\geq 1$  
\begin{align*}
\int_t^T (\varphi^k\cdot \alpha)(s) \uds =\int_{0}^T \alpha\big( [r-T,(r-t)\wedge 0) \big) |\varphi_r|^k  \ud r, \quad \forall t\in [0,T], \; \IP-a.s.
\end{align*}
Moreover, if we have for $p \geq 1$ that $\varphi \in \cH^p_0$, then we also have that
\[ 
\| (\varphi \cdot \alpha) \|_{\cH^p_\beta}^p 
\leq
 M_p \| \varphi \|_{\cH^p_0}^p,
\] where $M_p = (e^{\beta T})^{p/2} \big( \alpha([-T,0)) \big)^p$.
\end{lemma}
\begin{proof}
Let $t$ in $[0,T]$ and $k\in[1,+\infty)$. We have that
\begin{align*}
\int_t^T (\varphi^k\cdot\alpha)(s) \uds&= \int_t^T \int_{-T}^0 |\varphi_{s+v}|^k \alpha(\ud v) \uds 
= \int_{-T}^0 \int_t^T |\varphi_{s+v}|^k ~\uds ~\alpha(\ud v)\\
&= \int_{-T}^0 \int_{(t+v) \vee 0}^{T+v} |\varphi_r|^k ~\ud r ~\alpha(\ud v) 
= \int_{0}^T \int_{(r-T)}^{(r-t) \wedge 0} |\varphi_r|^k  ~\alpha(\ud v) ~\ud r \\
& = \int_{0}^T \alpha\big( [r-T,(r-t)\wedge 0) \big) |\varphi_r|^k  \ud r.
\end{align*}
The second claim follows by applying Jensen's inequality and changing the integration order as done above, i.e. for any $\beta\geq 0$ and $p\geq 1$ we have
\begin{align*}
\IE\left[ \Big( \int_0^T e^{\beta s} \vert (\varphi \cdot \alpha)(s) \vert^2 \uds \Big)^{p/2}\right] &\leq \big(e^{\beta T}\alpha([-T,0)) \big)^{p/2} ~ \IE\left[ \Big( \int_0^T (|\varphi|^2 \cdot \alpha)(s) \uds \Big)^{p/2} \right]
\\
& 
\leq M_p \IE\left[ \Big( \int_0^T \vert \varphi_s \vert^2 \uds \Big)^{p/2} \right]
= M_p \|\varphi\|_{\cH_0^p}^p,
\end{align*}
which concludes the proof.
\end{proof}

\medskip

\section{General results on BSDE with time delayed generators}
\label{section:apriori}
In this section we give a brief overview of BSDEs with time delayed generators and discuss the setting they are studied under. We then establish convenient a priori estimates on the difference of two solutions to such equations which will play a central role in proving existence and uniqueness of solutions in the more general $\cH^p$-spaces.

\subsection{BSDEs with time delayed generators}\label{sec:delay_bsde}
Let us start with a recap on BSDE with time delayed generators. Throughout the text, we assume 
\begin{enumerate}
\item[(H0)] $\ay,\az$ are two non-random, finitely valued measures supported on $[-T,0)$
\end{enumerate}
We also define
\begin{align}
\label{eq:alpha1}
\alpha := \ay([-T,0)) \vee \az([-T,0)).
\end{align}
Given $p\geq 2$, we assume that the following holds:
\begin{enumerate}[(H1)]
\item $\xi$ is an $\cF_T$-measurable random variable which belongs to $L^p(\IR^m)$;
\item the generator $f:\Omega\times [0,T]\times \IR^m \times \IR^{m\times d} \to \IR^m$ is measurable, $\IF$-adapted and satisfies the following Lipschitz like condition: there exists a constant $K>0$ such that 
\begin{align*}
\big| f(t,y,z) - f(t,y',z')  \big|^2
&
\leq K \big( |y-y'|^2 + |z-z'|^2 \big)
\end{align*}
holds for $\ud \IP \otimes \udt$-almost all $(\omega,t) \in \Omega \times[0,T]$ and for every $(y,z),(y'z') \in \IR^m \times \IR^{m\times d}$;
\item $\IE \Big[ \big( \int_0^T |f(s,0,0)|^2 \uds \big)^{p/2} \Big] < \infty$;
\item $f(t,\cdot,\cdot)=0$ if $t <0$.
\end{enumerate}
Following the notation from equation \eqref{eq:notation1}, we write
\begin{align*}
(Y\cdot\ay)(t) = \int_{-T}^0 Y_{t+v} \ay(\ud v)\ \text{ and }\
 (Z\cdot\az)(t) = \int_{-T}^0 Z_{t+v} \az(\ud v), \quad 0\leq t \leq T,
\end{align*}
for some processes $(Y_t)_{t\in[0,T]}$ and $(Z_t)_{t\in[0,T]}$ satisfying appropriate integrability conditions. Assumption (H2) and Jensen's inequality then imply
\begin{align*}
\text{(H2')} \quad 
&\big| f\big(t,(Y\cdot\ay)(t),(Z\cdot\az)(t) \big) - f\big(t,(Y'\cdot\ay)(t),(Z'\cdot\az)(t) \big)  \big|^2
\nonumber\\
&\qquad\quad
\leq K \big\{ \big| \big((Y-Y')\cdot\ay\big)(t) \big|^2 + \big| \big((Z-Z')\cdot\az\big)(t) \big|^2 \big\}
\nonumber\\
&\quad\qquad 
\leq L \big\{ \big((Y-Y')^2\cdot\ay\big)(t) + \big((Z-Z')^2\cdot\az\big)(t)\big\},
\end{align*}
where $L:=K\alpha$ with the real number $\alpha$ given by \eqref{eq:alpha1}. The focus of our study are BSDE with time delayed generators which are of the type
\begin{align}
\label{eq:bsde1}
Y_t &= \xi + \int_t^T f\big(s,\Gamma(s)\big) \uds - \int_t^T Z_s \udws, \quad 0 \leq t \leq T,
\end{align}
where $\Gamma$ abbreviates for $t\in[0,T]$
\begin{align}
\label{eq:Gamma1}
\Gamma(t) := \Big( \int_{-T}^0 Y_{t+v} \ay(\ud v), \int_{-T}^0 Z_{t+v} \az(\ud v) \Big) = \Big( (Y\cdot\ay)(t),(Z\cdot\az)(t) \Big).
\end{align}

\begin{definition}[Solution of a Delay BSDE]
We say $(Y,Z)$ is a solution to the delay BSDE \eqref{eq:bsde1} if $(Y,Z)$ belongs to the space $\cS^p_0\times \cH^p_0$ and satisfies \eqref{eq:bsde1}.
\end{definition}

Using a fixed point argument, \cite{DelongImkeller} have shown that a BSDE of the type \eqref{eq:bsde1}-\eqref{eq:Gamma1} admits a unique solution if the parameters of the equation \eqref{eq:bsde1} are sufficiently small, i.e. if the Lipschitz constant $K>0$ or the terminal time $T>0$ satisfy a smallness condition. The following $L^2$-existence and uniqueness result is a straightforward modification of Theorem 2.1 from \cite{DelongImkeller}.
\begin{theorem}
\label{theo:DelongImkeller_thm2.1}
Let $p=2$ and assume that (H0)-(H4) are satisfied. For $\alpha$ defined as in \eqref{eq:alpha1}, assume that the non-negative constants $T$, $L=K\alpha$, $\beta$ are such that
\[
(8 T + \frac1\beta) L \int_{-T}^0 e^{-\beta u}\rho(\ud u)\max\{1,T\}<1,\quad \text{for }\rho\in\{\ay,\az\}.
\]
Then the delay BSDE \eqref{eq:bsde1}-\eqref{eq:Gamma1} has a unique solution $(Y,Z)\in \cS^2_\beta(\IR^m)\times \cH^2_\beta(\IR^{m\times d})$.
\end{theorem}
\begin{remark}
In \cite{DelongImkeller}, this result is proved for the one-dimensional case $d=m=1$. It is clear that by the nature of the fixed point argument, the proof is insensitive to the dimension of the equation.
\end{remark}

\begin{remark} Given that a compatibility condition is necessary in order to establish existence and uniqueness of solutions and moreover that we will be giving an extended version of it, all the proofs in this section are given with extra detail in order to better control the constants involved in each result. 
\end{remark}

\subsection{Moment and a priori estimates}

In Lemma 2.1 from \cite{DelongImkeller} the authors provide a priori estimates for the time delayed BSDE \eqref{eq:bsde1} which estimates the norms of the difference between the solution of two BSDE in terms of the terminal condition and the difference of the generators applied to the solution processes. More specifically, for $i\in\{1,2\}$ let $(Y^i,Z^i)$ be the solution of a BSDE with dynamics \eqref{eq:bsde1} with terminal condition $\xi^i$ and driver $f^i$ satisfying (H1)-(H4), then it holds that
\begin{align}
\label{eq:DI}
\nonumber
&\|Y^1-Y^2 \|^2_{\cH^2_\beta} + \| Z^1-Z^2 \|^2_{\cH^2_\beta}
\leq C_2 \Big\{ \IE \big[\, e^{\beta T} |Y_T^1-Y_T^2|^2\,  \big] 
\\
&\qquad \quad +  \IE \big[ \int_0^T e^{\beta s} |f^1(s,(Y^1 \cdot \alpha)(s),(Z^1 \cdot \alpha)(s))-f^2(s,(Y^2 \cdot \alpha)(s),(Z^2 \cdot \alpha)(s))|^2 \uds\,  \big] \Big\},
\end{align} 
where the authors assume that $\alpha$ is some deterministic measure on $[-T,0)$ with mass one. Thus Lemma 2.1 from \cite{DelongImkeller} establishes the a priori estimate \eqref{eq:DI} whose right hand side depends again on the solution of \emph{both} delay BSDE. In the context of \cite{DelongImkeller} such a result suffices to establish existence and uniqueness of solutions in $\cS^2_\beta \times \cH^2_\beta$ but the situation becomes more intricate when the same issues are considered on $\cS^p_\beta \times \cH^p_\beta$ for $p>2$. More precisely, we are not able to obtain an estimate similar to \eqref{eq:DI} when $p>2$. In addition, the study of differentiability of the solution (for both $p=2$ and $p>2$), made in Section \ref{section:diff}, requires a priori estimates where the right hand side of the estimate depends only on the problem's data: the differences between the terminal conditions and a quantity of the form $\delta_2f_s:=f^1(s,(Y^2 \cdot \alpha)(s),(Z^2 \cdot \alpha)(s))-f^2(s,(Y^2 \cdot \alpha)(s),(Z^2 \cdot \alpha)(s))$. For a clear view of the required estimates, compare for instance \eqref{eq:DI} with \eqref{eq:apriori_p2}. 

\smallskip\subsubsection*{Moment estimates - part I}

As a starting observation, we have that if \eqref{eq:bsde1} admits a solution $(Y,Z)$ in $\cH^p_\beta(\IR^m) \times \cH^p_\beta(\IR^{m\times d})$, then we also have that $Y \in \cS^p_\beta(\IR^m)$.
\begin{lemma} 
Let $\beta\geq0$, $p\geq 2$ and assume that (H0)-(H4) hold. If the delay BSDE \eqref{eq:bsde1} admits a solution $(Y,Z) \in \cH^p_\beta(\IR^m) \times \cH^p_\beta(\IR^{m\times d})$ then we have also that $Y \in \cS^p_{\beta}(\IR^m)$.
\end{lemma}
\begin{proof}
Throughout let $t\in[0,T]$ and $p\geq 2$. Since all $\beta$-norms are equivalent, it suffices to show the result for $\beta=0$. We drop the $\beta$-subscripts in the following. The pair $(Y,Z)$ satisfies
\begin{align*}
Y_t &= \xi + \int_t^T f\big(s,(Y\cdot\ay)(s),(Z\cdot\az)(s)\big) \uds  - \int_t^T Z_s \udws,
\end{align*}
and in turn we have 
\begin{align*}
\sup_{0\leq t \leq T}|Y_t| & \leq |\xi| + \int_0^T \big| f\big(s,(Y\cdot\ay)(s),(Z\cdot\az)(s)\big) \big| \uds + \sup_{0\leq t \leq T} \big|\int_t^T Z_s \udws \big|.
\end{align*}
Combining the fact of $Z\in\cH^p$ with the inequalities by Young, Doob and Burkholder-Davis-Gundy (BDG), we obtain 
\begin{align*} 
\IE \Big[ \Big(  \sup_{0\leq t \leq T} \big| \int_t^T Z_s \udws \big|^2 \Big)^{p/2} \Big]
& 
\leq 2^{p/2} ~ \IE \Big[ \Big( \big| \int_0^T Z_s \udws \big|^2 + \sup_{0\leq t \leq T} \big| \int_0^t Z_s \udws \big|^2  \Big)^{p/2} \Big]
\\
& 
\leq 2^p ~ \IE \Big[ \sup_{0\leq t \leq T} \big| \int_0^t Z_s \udws \big|^p \Big]
\leq 2^p C_p  \|Z\|_{\cH^p_0}^p 
< \infty.
\end{align*}
Next observe that by the Lipschitz property of the generator $f$ (notice that (H2) implies (H2')), it follows that
\begin{align*}
&
\Big(  \int_0^T \big| f\big(s,(Y\cdot\ay)(s),(Z\cdot\az)(s) \big) \big|^2 \uds  \Big)^{p/2}\\
&\quad \leq 2^{p/2} \Big( \int_0^T \big| f(s,0,0) \big|^2 \uds  
+ \int_0^T \Big| f\big(s,(Y\cdot\ay)(s),(Z\cdot\az)(s) \big) - f(s,0,0) \Big|^2 \uds \Big)^{p/2}
\\
&\quad
\leq 2^{p/2}2^{p/2-1} \bigg\{  \Big( \int_0^T \big| f(s,0,0) \big|^2 \uds  \Big)^{p/2}+ \Big(  L \int_0^T \Big( (|Y|^2\cdot\ay)(s) + (|Z|^2\cdot\az)(s) \big) \uds \Big)^{p/2} \bigg\}.
\end{align*}
The second term in the bracket can be further estimated by 
\begin{align*}
&\left(  L \int_0^T \Big( (|Y|^2\cdot\ay)(s) + (|Z|^2\cdot\az)(s)  \Big) \uds \right)^{p/2}
\\
&\qquad
\leq 2^{p/2-1}L^{p/2} \left\{ \Big( \int_0^T (|Y|^2\cdot\ay)(s) \uds \Big)^{p/2} + \Big( \int_0^T (|Z|^2\cdot\az)(s) \uds\Big)^{p/2} \right\}
\\
&\qquad 
\leq 2^{p/2-1} L^{p/2}\alpha^{p/2} \left\{ \Big( \int_0^T |Y_s|^2 \uds \Big)^{p/2} + \Big( \int_0^T |Z_s|^2 \uds\Big)^{p/2} \right\},
\end{align*}
where the last line follows from 
Lemma \ref{lemma:interchange}. This estimate together with 
(H3) yields 
\begin{align*}
\IE \Big[ \Big( \int_0^T \big| f\big(s,(Y\cdot\ay)(s),(Z\cdot\az)(s) \big) \big|^2 \uds  \Big)^{p/2} \Big] &< \infty. 
\end{align*}
Using hypothesis (H1), i.e. that $\xi$ is in $\in L^p$, we can conclude that $Y\in\cS^p$ must hold. 
\end{proof}

\subsubsection*{A priori estimates}

Let us define the weighted variant $\tilde\alpha$ of $\alpha$ as the maximum of the weighted measures $\ay$ and $\az$ on $[-T,0)$ by
\begin{align}
\label{eq:alpha_tilde}
\tilde\alpha := \int_{-T}^0 e^{-\beta s} \ay(\uds) \vee \int_{-T}^0 e^{-\beta s} \az(\uds), \quad \beta\geq 0.
\end{align}
\begin{remark}
We emphasize that $\tilde{\alpha}$ depends on $\beta$. To keep the notation to a minimum we simply write $\tilde{\alpha}$ instead of making the dependence explicit.
\end{remark}

The next results establishes \textit{canonical} a priori estimates (in the sense that the right hand side of the estimate only depends on the problem's data) for the solutions of two time-delayed BSDEs as given by \eqref{eq:bsde1}. We distinguish between the cases $p=2$ and $p>2$, and we start with the case $p=2$.
\begin{proposition}[A priori estimates for $p=2$]
\label{prop:apriori_p2}
Let $p=2$. Consider $i\in\{1,2\}$ and let $(Y^i,Z^i) \in\cS^2_0 \times \cH^2_0$ be the solution of the delay BSDE \eqref{eq:bsde1} with terminal condition $\xi^i$ and generator $f^i$ satisfying (H0)-(H4). Denote by $K>0$ the Lipschitz constant of $f^1$ as given in (H2') and set $\delta Y = Y^1 - Y^2$, $\delta Z = Z^1 - Z^2$. 
If either $T$ or $K$ or $\alpha$ are small enough then there exist two constants $\beta, \gamma>0$ satisfying
\begin{align}
\label{eq:consts}
D_1:=\beta - \gamma - \frac{\tilde{\alpha} L}{\gamma} > 0
\quad\text{and}\quad 
D_2:= 1- \frac{\tilde{\alpha} L }\gamma>0
\qquad (\text{with }L=K\alpha \text{ and }\alpha\text{ as in \eqref{eq:alpha1}}),
\end{align} 
and a constant $C_2 =C_2(\beta,\gamma,\tilde\alpha,L,T)>0$ depending on $\beta,\gamma,\tilde\alpha,L,T$ such that: 
for $i\in\{1,2\}$, $(Y^i,Z^i) \in\cS^2_\beta \times \cH^2_\beta$ and
\begin{align}
\label{eq:apriori_p2}
\|\delta Y \|^2_{\cS^2_\beta} + \|\delta Y \|^2_{\cH^2_\beta} + \| \delta Z \|^2_{\cH^2_\beta} &\leq C_2 \Big\{ \IE \Big[ e^{\beta T} |\delta Y_T|^2  \Big] +  \IE \Big[ \int_0^T e^{\beta s} |\delta_2 f_s|^2 \uds  \Big] \Big\},
\end{align}
where $\delta_2 f_t := f^1\big(t,(Y^2 \cdot \ay)(t),(Z^2 \cdot \ay)(t)\big) - f^2\big(t,(Y^2 \cdot \ay)(t),(Z^2 \cdot \ay)(t)\big)$ for $t\in[0,T]$.
\end{proposition}
\begin{proof}
Let $\gamma, K, T, \alpha$ be such that the relations in \eqref{eq:consts} are satisfied (i.e. $D_1>0$ and $D_2>0$). Throughout let $t\in[0,T]$, $i\in\{1,2\}$ and define $\Gamma^i$  as in \eqref{eq:Gamma1} for the pair $(Y^i,Z^i)$. An application of It\^o's formula to the semimartingale $e^{\beta t}|\delta Y_t|^2$ for $\beta>0$ yields
\begin{align*}
&e^{\beta t}|\delta Y_t|^2 + \int_t^T \beta e^{\beta s}|\delta Y_s|^2 \uds + \int_t^T e^{\beta s}|\delta Z_s|^2 \uds\\
&\qquad = e^{\beta T}|\delta Y_T|^2 + \int_t^T 2e^{\beta s} \big\langle \delta Y_s, f^1(s,\Gamma^1(s)) - f^2(s,\Gamma^2(s)) \big\rangle \uds - \int_t^T 2e^{\beta s} \langle \delta Y_s,\delta Z_s \udws \rangle \\
&\qquad \leq e^{\beta T}|\delta Y_T|^2 + \int_t^T \gamma e^{\beta s} |\delta Y_s|^2 \uds + \int_t^T \frac{e^{\beta s}}{\gamma} \Big( \big| f^1(s,\Gamma^1(s)) - f^1(s,\Gamma^2(s)) \big|^2 \Big) \uds\\
& \qquad\qquad  + 2 \int_t^T e^{\beta s} \big\langle \delta Y_s, \delta_2 f_s \big\rangle \uds -  \int_t^T 2e^{\beta s} \langle \delta Y_s,\delta Z_s \udws \rangle,
\end{align*} 
where the last inequality results from Young's inequality for $\gamma$. Reorganizing and taking condition (H2') for the generator $f^1$ into account, we get
\begin{align*}
\nonumber
&e^{\beta t}|\delta Y_t|^2 + \int_t^T (\beta-\gamma) e^{\beta s}|\delta Y_s|^2 \uds + \int_t^T e^{\beta s}|\delta Z_s|^2 \uds\\
\nonumber
&\qquad \leq e^{\beta T}|\delta Y_T|^2 +  \int_t^T \frac{e^{\beta s}}{\gamma} L \Big[ (|\delta Y|^2\cdot\ay)(s) + (|\delta Z|^2\cdot \az)(s) \Big] \uds \nonumber\\
\nonumber
& \qquad\qquad  + 2 \int_t^T e^{\beta s} \big\langle \delta Y_s, \delta_2 f_s \big\rangle \uds -  \int_t^T 2e^{\beta s} \langle \delta Y_s,\delta Z_s \rangle \udws. 
\end{align*}

By a change of integration order argument similar to that in the proof of Lemma \ref{lemma:interchange} we obtain for 
$j\in\{ {\scriptstyle{\mathcal{Y}}}, {\scriptstyle{\mathcal{Z}}}\}$ and $\phi^{\scriptscriptstyle{\mathcal{Y}}}=\delta Y$, $\phi^{\scriptscriptstyle{\mathcal{Z}}}=\delta Z$ 
\begin{align}
\label{eq:tmp_02}
&\int_t^T e^{\beta s} (|\phi^j|^2\cdot \alpha_j)(s) \uds
\nonumber \\
&\quad= \int_t^T \int_{-T}^0 e^{\beta (s+v)} e^{-\beta v} \1_{\{s+v \geq 0\}} |\phi^j_{s+v}|^2 \alpha_j(\ud v) \uds 
\nonumber \\
&\quad = \int_{-T}^0 \int_{(t+v)\vee 0}^{T+v} e^{\beta r} e^{-\beta v} \1_{\{r \geq 0\}} |\phi^j_r|^2 \ud r ~ \alpha_j(\ud v) 
= \int_0^T \int_{r-T}^{(r-t)\wedge 0} e^{\beta r} e^{-\beta v} |\phi^j_r|^2 \alpha_j(\ud v) ~ \ud r 
\nonumber \\
& \quad \leq \int_0^T e^{\beta r} |\phi^j_r|^2\big( \int_{-T}^0 e^{-\beta v} \alpha_j(\ud v) \big)\ud r
\leq \int_0^T \tilde{\alpha} e^{\beta r} |\phi^j_r|^2 \ud r,
\end{align}
with $\tilde{\alpha}$ given by \eqref{eq:alpha_tilde}. Continuing the inequality from above we get
\begin{align}
&e^{\beta t}|\delta Y_t|^2 + \int_t^T (\beta-\gamma) e^{\beta s}|\delta Y_s|^2 \uds + \int_t^T e^{\beta s}|\delta Z_s|^2 \uds \leq e^{\beta T}|\delta Y_T|^2 + 2 \int_t^T e^{\beta s} \big\langle \delta Y_s, \delta_2 f_s \big\rangle \uds \nonumber\\
&\qquad\qquad +  \int_0^T \frac{\tilde{\alpha} L}{\gamma} e^{\beta s}  \Big( |\delta Y_s|^2 + |\delta Z_s|^2 \Big) \uds -  \int_t^T 2e^{\beta s} \langle \delta Y_s,\delta Z_s \udws \rangle.\label{eq:auxi1}
\end{align}
Taking the expectations for $t=0$ yields
\begin{align*}
&\big(\beta - \gamma - \frac{\tilde\alpha L}{\gamma} \big) \IE \Big[ \int_0^T  e^{\beta s}|\delta Y_s|^2 \uds \Big] + \big(1-\frac{\tilde\alpha L}{\gamma} \big) \IE \Big[ \int_0^T  e^{\beta s}|\delta Z_s|^2 \uds  \Big]
\\
&\quad \leq \IE \Big[ e^{\beta T}|\delta Y_T|^2 \Big] + 2 \IE\Big[ \int_0^T e^{\beta s} \big\langle \delta Y_s, \delta_2 f_s \big\rangle \uds  \Big]
\\
&\quad \leq \IE \Big[ e^{\beta T}|\delta Y_T|^2 \Big] + 2 \IE\Big[ \sup_{0\leq t \leq T} e^{\frac{\beta}{2} t} \vert \delta Y_t\vert \int_0^T e^{\frac{\beta}{2} s} |\delta_2 f_s| \uds  \Big]
\\
&\quad \leq \IE \Big[ e^{\beta T}|\delta Y_T|^2 \Big] + \gamma' \IE\Big[ \sup_{0\leq t \leq T} e^{\beta t} \vert \delta Y_t\vert^2\Big] + \frac{1}{\gamma'} \IE\Big[\big(\int_0^T e^{\frac{\beta}{2} s} |\delta_2 f_s| \uds\big)^2  \Big]
\end{align*}
where we have used Young's inequality with some $\gamma'>0$ to be specified later. From the last expression and since $D_1,D_2>0$ (see \eqref{eq:consts}) we deduce that
\begin{align}\label{eq:tmp_01}
\| \delta Y \|^2_{\cH^2_\beta} + \| \delta Z \|^2_{\cH^2_\beta} &\leq C \Big\{ \IE \Big[ e^{\beta T} |\delta Y_T|^2  \Big] + \gamma' 
\|\delta Y\|_{\cS^2_\beta}^2
+ \frac{1}{\gamma'} \IE\Big[\big(\int_0^T e^{\frac{\beta}{2} s} |\delta_2 f_s| \uds\big)^2 \Big\},
\end{align}
where $C >0$ is a constant depending $\beta,\gamma,\tilde\alpha,L$ and $T$. In order to obtain the $\cS^2_\beta$-estimate for $\delta Y$ we observe that we have
\begin{align*} 
\delta Y_t &\leq \delta Y_T + \int_t^T \big|f^1\big(s,\Gamma^1(s)\big) - f^1\big(s,\Gamma^2(s)\big)  \big| \uds + \int_t^T \big| \delta_2 f_s \big| \uds - \int_t^T \delta Z_s \udws.
\end{align*}
Multiplying by the monotone increasing function $e^{\frac{\beta}{2} t}$ and taking the conditional expectation with respect to $\cF_t$ we get
\begin{align*}
e^{\frac\beta2 t} \delta Y_t &\leq  \IE \left[ e^{\frac\beta2 t}|\delta Y_T| + e^{\frac\beta2 t} \int_t^T \big|f^1\big(s,\Gamma^1(s)\big) - f^1\big(s,\Gamma^2(s)\big)  \big| \uds +  e^{\frac\beta2 t} \int_t^T  \big|\delta_2 f_s \big| \uds   \big| \cF_t \right]
\\
&
\leq \IE \left[ e^{\frac\beta2 T}|\delta Y_T| +  \int_t^T e^{\frac\beta2 s} \big|f^1\big(s,\Gamma^1(s)\big) - f^1\big(s,\Gamma^2(s)\big)  \big| \uds \right.
\\
&
\qquad \left. + \int_0^t e^{\frac\beta2 s} \big|f^1\big(s,\Gamma^1(s)\big) - f^1\big(s,\Gamma^2(s)\big)  \big| \uds +   \int_t^T e^{\frac\beta2 s}  \big|\delta_2 f_s \big| \uds + \int_0^t e^{\frac\beta2 s}  \big|\delta_2 f_s \big| \uds  \big| \cF_t \right]
\\
&= \IE \left[ e^{\frac\beta2 T}|\delta Y_T| +  \int_0^T e^{\frac\beta2 s} \big|f^1\big(s,\Gamma^1(s)\big) - f^1\big(s,\Gamma^2(s)\big)  \big| \uds  +  \int_0^T e^{\frac\beta2 s}  \big|\delta_2 f_s \big| \uds  \big| \cF_t \right].
\end{align*}
Using Doob's inequality, we obtain
\begin{align*}
&
\|\delta Y\|_{\cS^2_\beta}^2 
\\
&\leq 4 ~ \IE \Big[ \Big(\IE \Big[ e^{\frac\beta2 T}|\delta Y_T| +  \int_0^T e^{\frac\beta2 s} \big|f^1\big(s,\Gamma^1(s)\big) - f^1\big(s,\Gamma^2(s)\big)  \big| \uds 
+  \int_0^T e^{\frac\beta2 s}  \big|\delta_2 f_s \big| \uds   ~ \big| ~ \cF_T\Big]\Big)^2 \Big]
\\
&
\leq 12 ~ \IE \Big[ e^{\beta T}|\delta Y_T|^2 +  T \int_0^T e^{\beta s} \big|f^1\big(s,\Gamma^1(s)\big) - f^1\big(s,\Gamma^2(s)\big)  \big|^2 \uds 
  + \big(\int_0^T e^{\frac{\beta}{2} s} \big|\delta_2 f_s \big| \uds\big)^2 \Big],
\end{align*}
where the last line follows by Jensen's inequality. Since $f^1$ satisfies (H2'), an application of Lemma \ref{lemma:interchange} yields
\begin{align*}
\|\delta Y\|_{\cS^2_\beta}^2
&
\leq 12\Big\{ \IE \Big[ e^{\beta T}|\delta Y_T|^2 \Big] + \tilde\alpha T L \Big( \| \delta Y \|^2_{\cH^2_\beta} + \| \delta Z \|^2_{\cH^2_\beta} \Big)+ \E\Big[ \big(\int_0^T e^{\frac{\beta}{2} s}  \big|\delta_2 f_s \big| \uds\big)^2\Big] \Big\}.
\end{align*}
Hence, plugging into \eqref{eq:tmp_01} we find 
\begin{align*}
&\big(1-12 C \gamma' \tilde{\alpha} T L\big)
\IE \Big[ \sup_{0\leq t \leq T} e^{\beta t} |\delta Y_t|^2 \Big]
\\
&\qquad \leq 12\Big\{ \big(1+C \tilde{\alpha} T L\big) \IE \Big[ e^{\beta T}|\delta Y_T|^2 \Big]
+  \big(1+C \gamma'^{-1} \tilde{\alpha} T L\big) \IE\Big[\big(\int_0^T e^{\frac{\beta}{2} s}  \big|\delta_2 f_s \big| \uds\big)^2 \Big] \Big\}.
\end{align*}
Choosing $\gamma'$ small enough such that $(1-12 C \gamma' \tilde{\alpha} T L)>0$ is satisfied we conclude that estimate \eqref{eq:apriori_p2} holds for a constant $C_2 = C_2(\beta,\gamma,\tilde\alpha,L,T)$.
\end{proof}
\begin{remark}
Note that in the previous result we have three degrees of freedom: the Lipschitz constant of the driver $K$, the time horizon $T$ and the duration of the time delay given by $\alpha$.  
\end{remark}

The proof for the case $p>2$ is more involved and uses techniques from the proof of Proposition \ref{prop:apriori_p2}. The main reason for the proof to be more involved can be seen in \eqref{eq:auxi1}. Usually the dynamics of $Y_t$ is described by integrals over the interval $[t,T]$ but for delay BSDEs we see from \eqref{eq:auxi1} that the dynamics of $Y_t$ depends also on a integral over the whole interval $[0,T]$. We also remark that the techniques of \cite{DelongImkeller} cannot be extended in $L^p$ (for $p>2$), see for instance estimate (2.3) present in the proof of Lemma 2.1 in \cite{DelongImkeller}.

\smallskip

The next proposition gives a result that will be central in establishing existence and uniqueness of $L^p$-solutions to delay BSDEs as well as in proving the differentiability results of  Section \ref{section:diff}.

\begin{proposition}[A priori estimates for $p>2$]
\label{lemma:apriori}
Let $p> 2$. Consider $i\in\{1,2\}$ and denote by $(Y^i,Z^i)\in\cS^p_0 \times \cH^p_0$ a solution of the delay BSDE \eqref{eq:bsde1} with terminal condition $\xi^i$ and generator $f^i$ satisfying (H0)-(H4). Denote by $K>0$ the Lipschitz constant of $f^1$ in (H2') and set $\delta Y = Y^1 - Y^2$, $\delta Z = Z^1 - Z^2$. If either $T$ or $K$ or $\alpha$ are small enough (for $L=K\alpha$, $\alpha$ as in \eqref{eq:alpha1} and $\tilde\alpha$ as in \eqref{eq:alpha_tilde}) then  there exists $\beta, \gamma>0$ satisfying \eqref{eq:consts} (i.e. $D_1,D_2>0$) and
\begin{align}
\label{eq:D_4}
D_3 &:= 1 - 2^{4p - 4} d^2_{p/2} \big(\frac{p}{p-2}\big)^{p/2}  \big(\frac{\tilde{\alpha} L}{\gamma - \tilde{\alpha} L}\big)^{p/2} D_2^{-p/2}- \big(\frac{\tilde{\alpha} L}{\gamma}T \big)^{p/2} \big(\frac{p}{p-2}\big)^{p/2}2^{p-2}>0
\end{align}
where $m\in \IN$ denotes the dimension of the $\delta Y$ process and the constant $d_{p/2}$ is given by
\begin{align}
\label{BDGconstant}
d_{p/2} := m^{p/2+1}\big(\frac{p}{p-1}\big)^{p^2/2} \Big( \frac{p(p-1)}{2} \Big)^{p/2}.
\end{align}
In addition, $(Y^i,Z^i) \in\cS^p_\beta \times \cH^p_\beta$ ($i\in\{1,2\}$) and there exists a constant $C_p = C_p(\beta,\gamma,\tilde\alpha,L,T,m) > 0$ explicitly given in \eqref{eq:dummy_whatever} such that
\begin{align}
\label{eq:apriori_p}
\|\delta Y \|^p_{\cS^p_\beta} + \|\delta Y \|^p_{\cH^p_\beta} + \| \delta Z \|^p_{\cH^p_\beta}
&\leq C_p \Big\{ \IE \Big[ \big(e^{\beta T} |\delta Y_T|^2 \big)^{p/2} \Big] +  \IE \Big[ \big(\int_0^T e^{\frac{\beta}{2} s}|\delta_2 f_s| \uds \big)^p \Big] \Big\},
\end{align}
with
$\delta_2 f_t = f^1\big(t,Y^2(t),Z^2(t)\big) - f^2\big(t,Y^2(t),Z^2(t)\big)$, for $t\in[0,T]$. 
\end{proposition}
\begin{remark}
\label{remark1onhowsmallKTalphaare}
A closer analysis on the constants $D_1$, $D_2$ and $D_3$ shows: 
\[
\lim_{K \alpha \to 0} (D_1,D_2,D_3)> (0,0,0).
\]
This means that with either a small $T$ or a small $K$ or a small $\alpha$ the conditions of the previous result can be verified.
\end{remark}
\textit{Proof of Proposition \ref{lemma:apriori}.}
Throughout let $t\in[0,T]$, $i\in\{1,2\}$ and from \eqref{eq:consts} define $D_1 := \beta - \gamma - \frac{\tilde\alpha L}{\gamma}$  and $D_2 := 1 - \frac{\tilde\alpha L}{\gamma}$. We emphasize that $\tilde{\alpha}$ as defined in \eqref{eq:alpha_tilde} depends on $\beta$. Recall \eqref{eq:auxi1} from the proof of Proposition \ref{prop:apriori_p2}: 
\begin{align}\label{eq:dummy1}
&e^{\beta t}|\delta Y_t|^2 + \int_t^T (\beta-\gamma) e^{\beta s}|\delta Y_s|^2 \uds + \int_t^T e^{\beta s}|\delta Z_s|^2 \uds \leq e^{\beta T}|\delta Y_T|^2 + 2 \int_t^T e^{\beta s} \big\langle \delta Y_s, \delta_2 f_s \big\rangle \uds \nonumber\\
&\qquad\qquad +  \int_0^T \frac{\tilde{\alpha} L}{\gamma} e^{\beta s}  \Big( |\delta Y_s|^2 + |\delta Z_s|^2 \Big) \uds -  \int_t^T 2e^{\beta s} \langle \delta Y_s,\delta Z_s \udws \rangle.
\end{align} 
By assumption $\beta,\gamma,T,K,\alpha$ are such that \eqref{eq:consts} holds and hence we have that $D_1>0$ and $D_2>0$. We carry out the proof in several steps.

\smallskip

\textsf{Step 1:} We claim that 
\begin{align}
&\IE \left[ \Big( \int_0^T e^{\beta s}|\delta Z_s|^2 \uds \Big)^{p/2}  \right]
 \leq D_2^{-p/2} \Big\{ 2^{p/2} \IE \Big[  \big( e^{\beta T}|\delta Y_T|^2 \big)^{p/2} \Big] 
+ 2^{3p -2} d^2_{p/2} D_2^{-p/2}\|\delta Y\|_{\cS^p_\beta}^p 
\nonumber\\
& \hspace{5cm} 
 + 2^{3p/2 - 1} \IE \Big[\,\big|\int_0^T e^{\beta s} \big\langle \delta Y_s, \delta_2 f_s \big\rangle \uds \big|^{p/2}\Big]
 \Big\}\label{eq:auxi4a},
\end{align}
where $d_{p/2} > 0$ is a given constant appearing in the BDG inequality which only depends on $p > 2$ and the dimension. Estimate \eqref{eq:auxi4a} can be deduced as follows: putting $t=0$ in \eqref{eq:dummy1} and noticing that by \eqref{eq:consts} the constants $D_1$ and $D_2$ are positive we get
\begin{align*}
\big(1-\frac{\tilde\alpha L}{\gamma} \big) \int_0^T e^{\beta s}|\delta Z_s|^2 \uds &\leq \big(\beta-\gamma-\frac{\tilde\alpha L}{\gamma} \big) \int_0^T e^{\beta s}|\delta Y_s|^2 \uds +  \big(1-\frac{\tilde\alpha L}{\gamma} \big) \int_0^T e^{\beta s}|\delta Z_s|^2 \uds\\
&\leq e^{\beta T}|\delta Y_T|^2 + 2 \int_0^T e^{\beta s} \big\langle \delta Y_s, \delta_2 f_s \big\rangle \uds - 2 \int_0^T e^{\beta s} \langle \delta Y_s,\delta Z_s \udws  \rangle.
\end{align*}
Now raising both sides to the power $p/2>1$, making use of the fact that for $a,b,c \in \IR$ 
\begin{align*}
\big|a+2b-2c\big|^{p/2} 
&
\leq 2^{p/2-1} \Big( |a|^{p/2} + |2b-2c|^{p/2} \Big)
\leq 2^{p/2-1} \Big( |a|^{p/2} + 2^{p/2-1} \big( |2b|^{p/2} + |2c|^{p/2} \big) \Big) 
\\
&
= 2^{p/2 - 1} |a|^{p/2} + 2^{3p/2-2}|b|^{p/2} + 2^{3p/2-2}|c|^{p/2}  
\end{align*}
and taking expectations, we get
\begin{align}
&\big(1-\frac{\tilde\alpha L}{\gamma} \big)^{p/2} ~ \IE \Big[ \Big( \int_0^T e^{\beta s}|\delta Z_s|^2 \uds \Big)^{p/2}  \Big] 
\leq 2^{p/2-1} \IE \Big[ \big( e^{\beta T}|\delta Y_T|^2 \big)^{p/2} \Big]
\nonumber\\
& \hspace{1cm}
 + 2^{3p/2 - 2} \IE \Big[\, \Big|\int_0^T e^{\beta s} \big\langle \delta Y_s, \delta_2 f_s \big\rangle \uds \Big|^{p/2}\Big] 
 + 2^{3p/2 - 2} \IE \Big[\, \Big|\int_0^T e^{\beta s} \langle \delta Y_s,\delta Z_s \udws \rangle \Big|^{p/2} \Big] \label{eq:auxi2}.
\end{align}
Denoting $$ \ud N^j_t := \sum_{k=1}^d \delta Z^{k,j}_t \ud W^k_t,$$
we apply the BDG inequality with the constant
\[
C^* := \big(\frac{p}{p-1}\big)^{p^2/2} \Big( \frac{p(p-1)}{2} \Big)^{p/2} >0,
\]
(see Theorem 3.9.1 from \cite{Khosh} and solution to Problem 3.29, p. 231, in \cite {KaratzasShreve}) and Young's inequality with some constant $\gamma_2>0$ and obtain 
\begin{align}
&\IE \Big[ \Big|\int_0^T e^{\beta s} \langle \delta Y_s,\delta Z_s \udws \rangle \Big|^{p/2}  \Big] 
\leq \IE \Big[ \Big( \sum_{j=1}^m \big|\int_0^T e^{\beta s} \delta Y^j_s ~ \ud N^j_s \Big)^{p/2}  \Big] 
\nonumber \\
& \qquad
\leq m^{p/2} ~ \sum_{j=1}^m \IE \Big[ \big| \int_0^T e^{\beta s} \delta Y^j_s ~ \ud N^j_s \big|^{p/2}   \Big]
\leq C^* m^{p/2} ~ \sum_{j=1}^m \IE \Big[ \int_0^T e^{2\beta s} |\delta Y^j_s|^2 ~ \ud  \langle N^j \rangle_s \big|^{p/4}   \Big]
\nonumber\\
&
\qquad\leq C^* m^{p/2} ~ \sum_{j=1}^m \IE \Big[ \big( \sup_{0\leq t \leq T} e^{\beta t} |\delta Y^j_t|^2 \big)^{p/4} ~ \big(\int_0^T e^{\beta s} ~ \ud  \langle N^j \rangle_s \big)^{p/4}   \Big] 
\nonumber\\
&\qquad
\leq C^* m^{p/2} ~ \sum_{j=1}^m \Big( \gamma_2 \IE \Big[ \big( \sup_{0\leq t \leq T} e^{\beta t} |\delta Y^j_t|^2 \big)^{p/2} \Big] + \frac1\gamma_2 \IE \Big[ \big(\int_0^T e^{\beta s} ~ \ud  \langle N^j \rangle_s \big)^{p/2} \Big] \Big)
\nonumber\\
&\qquad
\leq C^* m^{p/2} ~ \Big( \gamma_2 \| \delta Y \|^p_{\cS^p_\beta} + \frac{m}{\gamma_2}  \IE \Big[ \Big( \sum_{j=1}^m\int_0^T e^{\beta s} ~ \ud  \langle N^j \rangle_s \Big)^{p/2} \Big] \Big) 
\nonumber\\
&\qquad
\leq C^* m^{p/2+1} ~ \left( \gamma_2 \| \delta Y \|^p_{\cS^p_\beta} + \frac1\gamma_2 \| \delta Z \|^p_{\cH^p_\beta} \right) 
\leq d_{p/2} ~ \Big\{ \gamma_2 \| \delta Y \|^p_{\cS^p_\beta} + \frac{1}{\gamma_2}\|\delta Z\|_{\cH^p_\beta}^p
\Big\},
\label{eq:auxi3}
\end{align}
where by \eqref{BDGconstant} we have that $C^*m^{p/2+1}=d_{p/2}$.
With the particular choice of 
\[
\gamma_2 := 2^{3p/2-1} d_{p/2} D_2^{-p/2} = 2^{3p/2-1} d_{p/2}\left( \frac{\gamma}{\gamma-\tilde{\alpha}L} \right)^{p/2} > 0,
\] plugging \eqref{eq:auxi3} into \eqref{eq:auxi2} yields 
\begin{align*}
&\Big( \big(1 - \frac{\tilde{\alpha}L}{\gamma} \big)^{p/2} - \frac{2^{3p/2 -2 }}{\gamma_2} d_{p/2} \Big) ~
\| \delta Z \|_{\cH^p_\beta}^p = \frac12 D_2^{p/2} ~ \| \delta Z \|_{\cH^p_\beta}^p \leq\\
&\quad \leq 2^{p/2 - 1} \IE \Big[ \big( e^{\beta T}|\delta Y_T|^2 \big)^{p/2} \Big] + 2^{3p/2-2} \IE \Big[\Big|\int_0^T e^{\beta s} \big\langle \delta Y_s, \delta_2 f_s \big\rangle \uds \Big|^{p/2}\Big] 
+ 2^{3p/2 - 2} d_{p/2} \gamma_2 
\|\delta Y\|_{\cS^p_\beta}^p,
\end{align*}
which implies the claim.

\smallskip

\textsf{Step 2:} We claim that 
\begin{align}
\label{eq:dummy2}
D_3 \|\delta Y\|_{\cS^p_\beta}^p
&\quad 
\leq \big(\frac{p}{p-2} \big)^{p/2} ~  \Big\{ \big( 2^{p-2} + 2^{3p/2-2} \big( \frac{\tilde\alpha L}{\gamma - \tilde\alpha L} \big)^{p/2}  \big) \IE \Big[ \big(e^{\beta T}|\delta Y_T|^2\big)^{p/2} \Big]
\nonumber\\
&\qquad
+ \big( 2^{3p/2-2} + 2^{5p/2 - 3} \big( \frac{\tilde\alpha L}{\gamma - \tilde\alpha L} \big)^{p/2} \big) ~ \IE \Big[\Big(\int_0^T e^{\beta s} \big| \big\langle \delta Y_s, \delta_2 f_s \big\rangle \big| \uds \Big)^{p/2}\Big] \Big\},
\end{align}
holds for
\begin{align}
\label{constant D4}
D_3 &:= 1 - 2^{4p - 4} d^2_{p/2} \big(\frac{p}{p-2}\big)^{p/2}  \big(\frac{\tilde\alpha L}{\gamma - \tilde\alpha L}\big)^{p/2} D_2^{-p/2}- \big(\frac{\tilde\alpha L}{\gamma}T \big)^{p/2} \big(\frac{p}{p-2}\big)^{p/2}2^{p-2}. 
\end{align}
Note that the choice of $K, T$ and $\alpha$ has been such that $D_3 > 0$ is satisfied. To prove \eqref{eq:dummy2}, we go back to \eqref{eq:dummy1}, where we take the conditional expectation with respect to $\cF_t$, then the supremum over $t\in[0,T]$, raise to the power $p/2$ and finally apply Doob's inequality to obtain 
\begin{align}
&\IE \Big[ \sup_{0\leq t \leq T} \big(e^{\beta t}|\delta Y_t|^2 \big)^{p/2} \Big]\nonumber \\
&\quad
\leq \IE \Big[ \sup_{0\leq t \leq T} \Big(\IE \Big[ e^{\beta T} |\delta Y_T|^2 + 2 \int_0^T e^{\beta s} \big| \big\langle \delta Y_s, \delta_2 f_s \big\rangle \big| \uds
\nonumber \\
&\hspace{4.8cm}
+ \int_0^T \frac{\tilde\alpha L}{\gamma} e^{\beta s} \big(|\delta Y_s|^2 + |\delta Z_s|^2 \big) \uds \big|\cF_t\Big]\Big)^{p/2} \Big]
\nonumber \\
& \quad 
\leq \big(\frac{p}{p-2} \big)^{p/2}  ~ \Big\{  2^{p-2} \IE \Big[ \big( e^{\beta T}|\delta Y_T|^2 \big)^{p/2} \Big] + 2^{3p/2-2} \IE \Big[ \big( \int_0^T e^{\beta s} \big| \big\langle \delta Y_s, \delta_2 f_s \big\rangle \big| \uds \big)^{p/2} \Big] 
\nonumber \\
&	
\hspace{2cm} 
+ 2^{p-2} \IE \Big[ \big( \int_0^T \frac{\tilde\alpha L}{\gamma}e^{\beta s}|\delta Y_s|^2 \uds \big)^{p/2} \Big] + 2^{p-2} \IE \Big[ \big( \int_0^T \frac{\tilde\alpha L}{\gamma}e^{\beta s}|\delta Z_s|^2 \uds \big)^{p/2} \Big] \Big\}\label{eq:auxi4b}.
\end{align}
Note that we made use of the fact that for $a,b,c,d \in \IR$ and $p>2$, we have
\begin{align*}
\big| a + 2b + c + d \big|^{p/2} &\leq 2^{p/2-1} \big( |a+2b|^{p/2} + |c+d|^{p/2}  \big)
\\
&
\leq 2^{p-2} |a|^{p/2} + 2^{3p/2-2}|b|^{p/2} + 2^{p-2} |c|^{p/2} + 2^{p-2} |d|^{p/2}. 
\end{align*}
Plugging \eqref{eq:auxi4a} into \eqref{eq:auxi4b}, we get
\begin{align*}
\|\delta Y\|_{\cS_\beta^p}^p
&\quad
\leq \big(\frac{p}{p-2} \big)^{p/2}  \bigg\{ 2^{p-2}\IE \Big[ \big( e^{\beta T}|\delta Y_T|^2 \big)^{p/2} \Big] + 2^{3p/2 - 2} \IE \Big[ \big( \int_0^T e^{\beta s} \big| \big\langle \delta Y_s, \delta_2 f_s \big\rangle \big| \uds \big)^{p/2} \Big]
\nonumber\\
&\qquad  + 2^{p-2} \big(\frac{\tilde\alpha L}{\gamma}\big)^{p/2}
\|\delta Y\|_{\cH^p_\beta}^{p}
+ \big( \frac{\tilde\alpha L}{\gamma} \big)^{p/2} ~ D_2^{-1} \times 2^{p-2} \Big\{ 2^{p/2}\IE \Big[ \big( e^{\beta T}|\delta Y_T|^2 \big)^{p/2} \Big]
\\
&\qquad  + 2^{3p/2-1} \IE \Big[\big(\int_0^T e^{\beta s} \big| \big\langle \delta Y_s, \delta_2 f_s \big\rangle \big| \uds \big)^{p/2}\Big]   + 2^{3p-2}  d^2_{p/2}D_2^{-p/2} \| \delta Y \|^p_{\cS^p_\beta} \Big\} \bigg\}\\
&\quad
\leq \big(\frac{p}{p-2} \big)^{p/2}  ~ \bigg\{ \Big(2^{p-2} + 2^{3p/2-2} \big( \frac{\tilde\alpha L}{\gamma - \tilde\alpha L} \big)^{p/2}\Big)\IE \Big[ \big( e^{\beta T}|\delta Y_T|^2 \big)^{p/2} \Big]  
\nonumber\\
&\qquad + \big( 2^{3p/2-2} + 2^{5p/2-3} \big( \frac{\tilde\alpha L}{\gamma - \tilde\alpha L} \big)^{p/2} \big)\IE \Big[\big(\int_0^T e^{\beta s} \big| \big\langle \delta Y_s, \delta_2 f_s \big\rangle \big| \uds \big)^{p/2}\Big] 
\\
&\qquad + \Big( 2^{p-2}\big( \frac{\tilde\alpha L}{\gamma} T \big)^{p/2} + 2^{4p-4}\big( \frac{\tilde\alpha L}{\gamma - \tilde\alpha L} \big)^{p/2} D_2^{-p/2} d^2_{p/2} \Big)\| \delta Y \|^p_{\cS^p_\beta}		\bigg\},
\end{align*}
from which the estimate \eqref{eq:dummy2} follows. 

\smallskip

\textsf{Step 3:} At this stage, estimating $\IE \big[\Big(\int_0^T e^{\beta s} \big| \big\langle \delta Y_s, \delta_2 f_s \big\rangle \big| \uds \Big)^{p/2}\Big]$ will yield \eqref{eq:apriori_p}. This itself is a consequence of \eqref{eq:dummy2}: Young's inequality combined with the $\cS^p_\beta$-norm yields
\begin{align}
\label{eq:dummy7}
 \IE \big[\Big(\int_0^T e^{\beta s} \big| \big\langle \delta Y_s, \delta_2 f_s \big\rangle \big| \uds \Big)^{p/2}\Big]
&\leq \IE \Big[\big(\int_0^T e^{\beta s} | \delta Y_s | ~ |\delta_2 f_s | \uds \big)^{p/2}\Big] \nonumber\\
&
\leq \gamma_3 
\| \delta Y \|_{\cS^p_\beta}^p
+ \frac{1}{\gamma_3} \IE\Big[\Big(\int_0^T e^{\frac{\beta}{2} s}\big |\delta_2 f_s| \uds \Big)^p\Big],
\end{align}
which in conjunction with the particular choice 
\begin{equation}
\label{eq:gamma3}
\gamma_3 := \frac12 D_3 \big(\frac{p-2}{p}\big)^{p/2} \frac{(\gamma - \tilde\alpha L)^{p/2}}{2^{3p/2 - 2}(\gamma - \tilde\alpha L)^{p/2} + 2^{5p/2 - 3} (\tilde\alpha L)^{p/2} } > 0.
\end{equation}
Estimate \eqref{eq:dummy2} now leads to
\begin{align}
\label{eq:dummy8}
&\frac12 D_3\| \delta Y \|^p_{\cS^p_\beta} 
 \leq \big(\frac{p}{p-2} \big)^{p/2} ~  \Big\{ \big( 2^{p-2} + 2^{3p/2-2} \big( \frac{\tilde\alpha L}{\gamma - \tilde\alpha L} \big)^{p/2}  \big) \IE \big[ \big(e^{\beta T}|\delta Y_T|^2\big)^{p/2} \big]
\nonumber\\
&\hspace{2.5cm}+ \big( 2^{3p/2-2} + 2^{5p/2 - 3} \big( \frac{\tilde\alpha L}{\gamma - \tilde\alpha L} \big)^{p/2} \big) \gamma_3^{-1} ~ \IE \big[\big(\int_0^T e^{\frac{\beta}{2} s} |\delta_2 f_s| \uds \big)^p\big] \Big\}.
\end{align}
Notice that we trivially have $\| \delta Y \|^p_{\cH^p_\beta} \leq T^{p/2} ~ \| \delta Y \|^p_{\cS^p_\beta}$ so that
\begin{align*}
&
\| \delta Y \|^p_{\cS^p_\beta} + \| \delta Y \|^p_{\cH^p_\beta} 
\leq
C^1_p ~ \IE \big[ \big(e^{\beta T}|\delta Y_T|^2\big)^{p/2} \big] + C^2_p ~ \IE \big[\big(\int_0^T e^{\frac{\beta}{2} s} |\delta_2 f_s| \uds \big)^p\big],
\end{align*}
where the constants $C^1_p$ and $C^2_p$ are defined as
\begin{align*}
&
C^1_p:=
2 (1 + T^{p/2}) D_3^{-1}\big(\frac{p}{p-2}\big)^{p/2} 
 \big( 2^{p-2} + 2^{3p/2-2} \big( \frac{\tilde\alpha L}{\gamma - \tilde\alpha L} \big)^{p/2}  \big),
\\
&
C^2_p:=
 2 (1 + T^{p/2}) D_3^{-1}\big(\frac{p}{p-2}\big)^{p/2}
\big( 2^{3p/2-2} + 2^{5p/2 - 3} \big( \frac{\tilde\alpha L}{\gamma - \tilde\alpha L} \big)^{p/2} \big) \gamma_3^{-1} .
\end{align*}
Moreover, it follows from \eqref{eq:auxi4a}, \eqref{eq:dummy7} and \eqref{eq:dummy8} that 
\begin{align*}
&\| \delta Z \|^p_{\cH^p_\beta} \leq C^3_p ~ \IE \big[ \big(e^{\beta T}|\delta Y_T|^2\big)^{p/2} \big] + C^4_p ~ \IE \big[\big(\int_0^T e^{\frac{\beta}{2} s} |\delta_2 f_s| \uds \big)^p\big].
\end{align*}
where the constants $C^3_p$ and $C^4_p$ are defined as
\begin{align*}
C^3_p:=&
2D_3^{-1}\big(\frac{p}{p-2}\big)^{p/2}  D_2^{-p/2}
\Big[ 2^{p/2} \\
&\hspace{3.2cm} + \big( 2^{3p-2}d^2_{p/2}D_2^{-p/2} + 2^{3p/2 - 1} \gamma_3 \big)  \big( 2^{p-2} + 2^{3p/2 -2} \big(\frac{\tilde\alpha L}{\gamma - \tilde\alpha L} \big)^{p/2} \big) \Big],
\\
C^4_p:=& 
2D_3^{-1}\big(\frac{p}{p-2}\big)^{p/2}  D_2^{-p/2}
\Big[ \big( 2^{3p-2} d^2_{p/2} D_2^{-p/2} + 2^{3p/2 -1} \gamma_3 \big)\times \\
&\hspace{3.2cm} \times \big( 2^{3p/2 -2} + 2^{5p/2-3} \big(\frac{\tilde\alpha L}{\gamma - \tilde\alpha L}\big) \big)^{p/2} \gamma_3^{-1} + 2^{3p/2 - 1} \gamma_3 \Big],
\end{align*}
(recall that $\gamma_3$ is defined by \eqref{eq:gamma3}).
From the above inequalities we obtain \eqref{eq:apriori_p}, where the positive constant $C_p$ is given by
\begin{align}
\label{eq:dummy_whatever}
C_p := \max \big\{ C_p^1 + C_p^3, C_p^2 + C_p^4  \big\}.
\end{align}
\hfill $\Box$

\begin{remark}
\label{constants Di independent of data}
Notice that none of the constants $C_p$, $C_p^i$ and $D_i$ ($i\in\{1,\cdots,4\}$) depend on the terminal condition or $f(\cdot,0,0)$. The only problem related data they do depend on are: $K$, $T$, $\alpha$ and $m$. 
\end{remark}

\begin{remark}
In the previous proof it is clear that our choices for the constants $\gamma_2$ and $\gamma_3$ do not lead to the most general statement of Proposition \ref{lemma:apriori}. They were chosen in this way to avoid a more complex statement, i.e. the constant $C_p$ given in \eqref{eq:dummy_whatever} would then depend on $\gamma_2$ and $\gamma_3$ and jointly with \eqref{eq:D_4} we would also have the condition $D_3>0$. The conditions of Theorem \ref{theo:picard} below depend on the smallness of $C_p$ as given by \eqref{eq:dummy_whatever}. The particular choices for $\gamma_2$ and $\gamma_3$ lead to simpler expressions in our statements.
\end{remark}

\subsubsection*{Moment estimates - part II}

As a by-product of the two previous propositions we obtain 
a result on the moment estimates for the solution of BSDE \eqref{eq:bsde1}.
\begin{corollary}[Moment estimates]
\label{coro:momentestimates}
Let $p \geq 2$ and $\beta>0$. Let $(Y,Z) \in\cS^p_\beta \times \cH^p_\beta$ be the solution of the delay BSDE \eqref{eq:bsde1} with terminal condition $\xi$ and generator $f$ satisfying (H0)-(H4). For $K,T,\alpha$ small enough, there exists a constant $C_p$ (which, like in Propositions \ref{prop:apriori_p2} and \ref{lemma:apriori}, depends on several constants that can be suitably chosen) such that  
\begin{align*}
\| Y \|^p_{\cS^p_\beta} + \| Y \|^p_{\cH^p_\beta} + \| Z \|^p_{\cH^p_\beta}
&
\leq C_p \Big\{ \IE \Big[ \big(e^{\beta T} | Y_T|^2 \big)^{p/2} \Big]
 +\IE \Big[ \big(\int_0^T e^{\beta s}| f(s,0,0)|^2 \uds \big)^p \Big] \Big\}.
\end{align*}
\end{corollary}

\subsubsection*{The existence and uniqueness result}

The moment and a priori estimates in \cite{DelongImkeller} are tailor-made for a Picard iteration procedure in $\cH^2\times \cH^2$. To make such a technique work in general $L^p$-spaces we needed to state a priori estimates in the form of Proposition \ref{prop:apriori_p2} and Proposition \ref{lemma:apriori}. In view of those results one can naturally expect a compatibility condition on $K,T$ and $\alpha$ more complicated than that of Theorem \ref{theo:DelongImkeller_thm2.1} for a solution to exist.
\vspace{0.3cm}

With estimate \eqref{eq:apriori_p} at hand, we now proceed to show the existence and uniqueness of solutions to \eqref{eq:bsde1} in $\cS^p_\beta \times \cH^p_\beta$ for $p>2$. For $p=2$, Theorem 2.1 from \cite{DelongImkeller} (recalled in our Theorem \ref{theo:DelongImkeller_thm2.1}) yields a sufficient condition which guarantees the standard Picard iteration to converge and proves the existence and uniqueness of solutions to \eqref{eq:bsde1}. We will show in the following result that for $p>2$, the convergence of the same Picard iteration is retained. What is needed to achieve this goal is to put up some extra effort to show that the Picard iterates $(Y^n,Z^n)$ satisfy the corresponding $\cS^p_\beta, \cH^p_\beta$-integrability properties.

\begin{theorem}\label{theo:picard} 
Let $p>2$ and assume that (H0)-(H4) hold. Let $K$ or $T$ or $\alpha$ be small enough such that for some $\beta,\gamma>0$ the conditions of Proposition  \ref{lemma:apriori} are satisfied. If further $K$ or $T$ or $\alpha$ are small enough such that we have  
\begin{align} 
\label{eq:contraction}
2^{p/2-1}C_p \Big(L T \int_{-T}^0 e^{-\beta s} \rho(\ud s) \Big)^{p/2} \max\{1,T^{p/2}\} < 1,\ \text{ for }\rho \in  \{\ay,\az \},
\end{align}
where $C_p = C_p(\beta,\gamma,\tilde\alpha,L,T,m) > 0$ is given by \eqref{eq:dummy_whatever}, $\tilde\alpha$ is given by \eqref{eq:alpha_tilde} and $L=K\alpha$, then the BSDE \eqref{eq:bsde1} admits a unique solution $(Y,Z)$ in $\cS^p_{\beta} \times \cH^p_{\beta}$.
\end{theorem}
\begin{remark}
Note that, by definition of the constant $C_p$, condition \eqref{eq:contraction} is satisfied if either $T$ or $K$ or $\alpha$ is small enough since $ \displaystyle{\lim_{T K \alpha \to 0} C_p <+\infty }$ which in turn implies
$$ \lim_{T K \alpha \to 0} C_p (\alpha K  T )^{p/2} =0.$$
\end{remark}
\textit{Proof of Theorem \ref{theo:picard}.}
Let $p>2$. 
Throughout let $t\in[0,T]$. The proof is based on the standard Picard iteration: we initialize by $Y^0 = 0$ and $Z^0=0$ and define recursively
\begin{align}
\label{eq:iterate}
Y^{n+1}_t &= \xi + \int_t^T f\big( s,\Gamma^n(s) \big) \uds - \int_t^T Z^{n+1}_s \udws, \quad 0 \leq t \leq T,
\end{align}
with 
$\Gamma^n(s) = \big(\int_{-T}^0 Y^n_{s+v} \ay(\ud v), \int_{-T}^0 Z^n_{s+v} \az(\ud v)  \big)$ for $s\in[0,T]$ and $n\in\IN$.
In the following, let $C>0$ denote some generic constant which may vary from line to line but is always independent of $n\in\IN$. We proceed by induction, where the existence of $(Y^1,Z^1)\in \cS^p_{\beta} \times \cH^p_{\beta}$ follows from classic stochastic analysis arguments. For $n\geq 1$, assume that $(Y^n,Z^n) \in \cS^p_{\beta} \times \cH^p_{\beta}$ solves the BSDE \eqref{eq:iterate} and we now prove that \eqref{eq:iterate} has a unique solution $(Y^{n+1},Z^{n+1}) \in \cS^p_{\beta} \times \cH^p_{\beta}$. Note that due to 
\begin{align}
&\IE \Big[ \big( \int_0^T |f(s,\Gamma^n(s))| \uds \big)^p \Big]
\nonumber \\
&\quad 
\leq \IE \Big[ \Big( \int_0^T |f(s,0,0)| \uds + \int_0^T |f(s,\Gamma^n(s))-f(s,0,0)| \uds \Big)^{p} \Big]
\nonumber\\
&\quad 
\leq 2^{p-1}~\IE \Big[ \Big( \int_0^T |f(s,0,0)| \uds \Big)^p +  \Big( T \int_0^T |f(s,\Gamma^n(s))-f(s,0,0)|^2 \uds \Big)^{p/2} \Big]
\nonumber\\
&\quad
\leq 2^{p-1}  ~ \IE \Big[ \big( \int_0^T |f(s,0,0)| \uds \big)^p
\nonumber\\
&\qquad + L^{p/2} T^{p/2} \Big\{  \int_0^T \int_{-T}^0 |Y^n_{s+v}|^2 \ay(\ud v) \uds + \int_0^T\int_{-T}^0 |Z^n_{s+v}|^2 \az(\ud v)\uds \Big\}^{p/2}  \Big]
\nonumber\\
&\quad
\leq 2^{p-1} \IE \Big[ \big( \int_0^T |f(s,0,0)| \uds \big)^p + (\alpha K T)^{p/2} \Big\{  \int_0^T |Y^n_{s}|^2 \uds + \int_{0}^T |Z^n_{s}|^2 \uds\Big\}^{p/2}  \Big]
\nonumber\\
&\quad
\leq 2^{p-1} \IE \Big[ \big( \int_0^T |f(s,0,0)| \uds \big)^p \Big] + 2^{p/2-1}(2 \alpha K T)^{p/2} \Big( T^{p/2} \| Y^n \|^p_{\cS^p_0} + \| Z^n \|^p_{\cH^p_0} \Big)
<\infty, \label{eq:dummy-gonca}
\end{align}
the martingale representation yields a uniquely determined process $Z^{n+1} \in \cH^2_0$ such that
\begin{align*}
\IE \Big[ \xi + \int_0^T f\big(s,\Gamma^n(s)\big) \uds \big| \cF_t \Big] = \IE \big[ \xi + \int_0^T f\big(s,\Gamma^n(s)\big) \uds \big] + \int_0^t Z^{n+1}_s \udws,\quad \text{for any }t\in[0,T].
\end{align*}
We then define $Y^{n+1}$ to be a continuous version of
$Y^{n+1}_t = \IE[ \xi + \int_t^T f(s,\Gamma^n(s)) \uds |\cF_t]$.
Let us first show that $Y^{n+1} \in \cS^p_{0}$:  
\begin{align*}
\| Y^{n+1}\|_ {\cS^p_0}^p
=  \IE \Big[ \sup_{t\in[0,T]} |Y^{n+1}_t|^p \Big]
&\leq \IE \Big[  \sup_{t\in[0,T]} \Big(\IE \big[\, |\xi| + \int_0^T |f(s,\Gamma^n(s))| \uds\,|\cF_t\big]\Big)^p  \Big]
\\
&
\leq \Big( \frac{p}{p-1} \Big)^{p} \IE \Big[  \Big( |\xi| + \int_0^T |f(s,\Gamma^n(s))| \uds \Big)^p  \Big]
\\
&
\leq 2^{p-1} \Big( \frac{p}{p-1} \Big)^{p} ~ \IE \Big[ |\xi|^p + \Big(\int_0^T |f(s,\Gamma^n(s))| \uds \Big)^p  \Big]
<\infty,
\end{align*}
where the last inequality follows from the fact that $\xi \in L^p$ and \eqref{eq:dummy-gonca}. This proves that $Y^{n+1} \in \cS^p_{0}$. Since all $\| \cdot \|_{\cS^p_\beta}$-norms are equivalent it follows that $Y^{n+1} \in \cS^p_{\beta}$. To see that $Z^{n+1} \in \cH^p_\beta$, recall that It\^o's formula applied to $e^{\beta t} |Y^{n+1}_t|^2$ yields
\begin{align*}
&e^{\beta t}|Y^{n+1}_t|^2 + \int_t^T \beta e^{\beta s} |Y^{n+1}_s|^2 \uds + \int_t^T e^{\beta s} |Z^{n+1}_s|^2 \uds\\
&\quad = e^{\beta T} |\xi|^2 + \int_t^T 2e^{\beta s} \langle Y^{n+1}_s,f(s,\Gamma^n(s))\rangle \uds - \int_t^T 2 e^{\beta s} \langle Y^{n+1}_s, Z^{n+1}_s \udws\rangle.
\end{align*}
In the above drop the two $Y$ terms in the LHS of the equation, take $t=0$, apply absolute values to both sides and then raise to power $p/2$. It follows that
\begin{align}
\label{eq:tmp1}
&
\big( \int_0^T e^{\beta s} |Z^{n+1}_s|^2 \uds \big)^{p/2}
\nonumber\\
& \quad 
\leq \Big(  e^{\beta T} |\xi|^2 + \int_0^T 2e^{\beta s}  |Y^{n+1}_s| ~ |f(s,\Gamma^n(s))| \uds + 
\big| \int_0^T 2 e^{\beta s} \langle Y^{n+1}_s, Z^{n+1}_s \udws \rangle  \big| \Big)^{p/2}
\nonumber\\
&\quad
\leq ~ 2^{p/2-1} \big(e^{\beta T} |\xi|^2\big)^{p/2} + 2^{p-2} \big(\int_0^T 2 e^{\beta s}  |Y^{n+1}_s| ~ |f(s,\Gamma^n(s))| \uds\big)^{p/2}
\nonumber\\
&\qquad 
+ 2^{3p/2-2}\big| \int_0^T  e^{\beta s} \langle Y^{n+1}_s, Z^{n+1}_s \udws \rangle   \big|^{p/2}.
\end{align}
On the one hand, we have
\begin{align}
\label{eq:tmp-between-tmp1and-tmp3}
&\IE \Big[ \Big( \int_0^T 2e^{\beta s} |Y^{n+1}_s| ~ |f(s,\Gamma^n(s))| \uds  \Big)^{p/2} \Big]
\nonumber
\\
&\quad 
\leq \IE \Big[ \Big( \int_0^T 2e^{\beta s} |Y^{n+1}_s| ~ |f(s,\Gamma^n(s))-f(s,0,0)| \uds + \int_0^T 2e^{\beta s} |Y^{n+1}_s| ~ |f(s,0,0)| \uds \Big)^{p/2} \Big]
\nonumber
\\
&\quad 
\leq C \Big\{ \| Y^{n+1} \|^p_{\cS^p_\beta} + \IE \Big[ \big( \int_0^T e^{\frac{\beta}{2} s}|f(s,0,0)| \uds \big)^p \Big] + \| Y^{n} \|^p_{\cS^p_\beta} + \| Z^{n} \|^p_{\cH^p_\beta} \Big\}
< \infty,
\end{align}
where we have used the Lipschitz condition of $f$ combined with calculations similar to those of \eqref{eq:dummy-gonca} and 
\[
\int_0^T 2e^{\beta s} |Y^{n+1}_s| ~ |f(s,0,0)| \uds 
\leq 
\sup_{0\leq t\leq T} e^{\beta t} |Y^{n+1}_t|^2 + \Big(\int_0^T e^{\frac{\beta}{2} s} |f(s,0,0)| \uds\Big)^2.
\]
On the other hand, by the same arguments as in \eqref{eq:auxi3} we find the following estimate
\begin{align}
\label{eq:tmp3}
\IE \Big[\,\big| \int_0^T e^{\beta s} \langle Y^{n+1}_s, Z^{n+1}_s \udws \rangle \big|^{p/2} \Big]
&
\leq d_{p/2} ~ \Big\{ \kappa~
\| Y^{n+1} \|_{\cS^p_\beta}^p
+ \frac1\kappa 
\| Z^{n+1} \|_{\cH^p_\beta}^p
\Big\},
\end{align}
where the last line the constant $\kappa>0$ appear due to Young's inequality. Now choosing $\kappa>0$ such that $1- 2^{2p-2}~d_{p/2}\kappa^{-1} > 0$, it follows from \eqref{eq:tmp1}, \eqref{eq:tmp-between-tmp1and-tmp3} and \eqref{eq:tmp3} that
\begin{align*}
\big(1- \frac{2^{2p-2}~d_{p/2}}{\kappa}\big)
\|Z^{n+1}\|_{\cH^p_\beta}^p
&
\leq C \Big\{ \IE \big[  \big(e^{\beta T} |\xi|^2 \big)^{p/2} \big] +  \| Y^{n+1} \|^p_{\cS^p_\beta} 
\\
&\hspace{1cm}
+ \IE \big[ \big( \int_0^T |f(s,0,0)| \uds \big)^p \big] + \| Y^{n} \|^p_{\cS^p_\beta} + \| Z^{n} \|^p_{\cH^p_\beta} \Big\}
< \infty.
\end{align*}
This proves that $Z^{n+1} \in \cH^p_\beta$.

\smallskip

In the next step, we prove that the sequence $(Y^n,Z^n)$ converges in  $\cS^p_\beta \times \cH^p_\beta$. Under the current assumptions one is able to apply a priori estimate \eqref{eq:apriori_p} to obtain
\begin{align*}
&\| Y^{n+1} - Y^n \|^p_{\cS^p_\beta} + \| Z^{n+1} - Z^n \|^p_{\cH^p_\beta}\\
&\quad \leq C_p ~ \IE \Big[ \Big( \int_0^T e^{\frac\beta2 s} \big| f(s,\Gamma^n(s)) - f(s,\Gamma^{n-1}(s)) \big| \uds \Big)^p \Big]\\
&\quad \leq C_p T^{p/2}~ \IE \Big[ \Big( \int_0^T e^{\beta s} \big| f(s,\Gamma^n(s)) - f(s,\Gamma^{n-1}(s)) \big|^2 \uds \Big)^{p/2} \Big].
\end{align*}
In analogy to the calculation carried out in Equation (2.7) in \cite{DelongImkeller}[Proof of Theorem 2.1], it is easy to see that we have
\begin{align*}
&\| Y^{n+1} - Y^n \|^p_{\cS^p_\beta} + \| Z^{n+1} - Z^n \|^p_{\cH^p_\beta}
\\
&\quad 
\leq C_p T^{p/2} ~ \IE \Big[ \Big( L \max\big\{ \int_{-T}^0 e^{-\beta s} \ay(\uds), \int_{-T}^0 e^{-\beta s} \az(\uds)  \big\}
\\
&\qquad\qquad  \times \big( T\sup_{t\in[0,T]} e^{\beta t}|Y^n_t-Y^{n-1}_t|^2 + \int_0^T e^{\beta s}|Z^n_s - Z^{n-1}_s|^2 \uds \big)  \Big)^{p/2} \Big]
\\
&\quad
\leq C_p T^{p/2} ~ 2^{p/2-1}~\Big(L \max\big\{ \int_{-T}^0 e^{-\beta s} \ay(\uds), \int_{-T}^0 e^{-\beta s} \az(\uds)  \big\}\Big)^{p/2}
\\
&\qquad\qquad
\times  \Big( T^{p/2} \| Y^n-Y^{n-1} \|^p_{\cS^p_\beta} + \| Z^n-Z^{n-1} \|^p_{\cH^p_\beta} \Big)
\\
&\quad
\leq C_p ~ 2^{p/2-1}~\Big(L T \max\big\{ \int_{-T}^0 e^{-\beta s} \ay(\uds), \int_{-T}^0 e^{-\beta s} \az(\uds)  \big\}\Big)^{p/2} ~ \max\big\{ 1,T^{p/2} \big\}
\\
&\qquad\qquad
\times \Big(  \| Y^n-Y^{n-1} \|^p_{\cS^p_\beta} + \| Z^n-Z^{n-1} \|^p_{\cH^p_\beta} \Big).
\end{align*}
Hence, by \eqref{eq:contraction}, the standard fixed point argument yields that $(Y^n,Z^n)$ converges in $\cS^p_\beta \times \cH^p_\beta$, which finishes the proof.
\hfill $\Box$

\medskip

\section{Decoupled FBSDE with time delayed generators}
\label{section:diff}
The objective of this section is to extend the results from \cite{DelongImkeller,DelongImkeller2} to the case of decoupled forward-backward stochastic differential equations. For measurable functions $b,\sigma,g,f$, specified in more detail below, we study the time delayed FBSDE
\begin{align}
\label{eq:fwd1}
X^x_t &= x + \int_0^t b(s,X^x_s) \uds + \int_0^t \sigma(s,X^x_s) \udws, \quad x \in \IR^d,\\
\label{eq:bwd1}
Y^x_t &= g(X^x_T) + \int_t^T f\big(s,\Theta^x(s)\big) \uds - \int_t^T Z^x_s \udws, \quad  0 \leq t \leq T,
\end{align}
where for $t\in[0,T]$, we write
\begin{align}
\nonumber
\Theta^x(t)
&=\big((X^x \cdot \ax)(t),(Y^x\cdot \ay)(t),(Z^x\cdot \az)(t)\big)\\
\label{eq:theta}
&=\Big(\int_{-T}^0 X^x_{t+v} \ax(\ud v),\int_{-T}^0 Y^x_{t+v} \ay(\ud v),\int_{-T}^0 Z^x_{t+v} \az(\ud v) \Big),
\end{align}
with given deterministic finite measures $\ax,\ay$ and $\az$ supported on $[-T,0)$. The coefficients $b,\sigma,g,f$ appearing in \eqref{eq:fwd1}-\eqref{eq:bwd1} are assumed to satisfy certain smoothness and integrability conditions such that the backward equation \eqref{eq:bwd1} falls back into the setting of (H0)-(H4) from Section \ref{sec:delay_bsde}. 
More precisely, we assume the following to hold:
\begin{enumerate}[(F1)]
\item[(F0)] $\ax$, $\ay,\az$ are three non-random, finitely valued measures supported on $[-T,0)$;
\item $g:\IR^d\to\IR^m$ is continuous differentiable with uniformly bounded first order derivatives, i.e. there exists $K'>0$ such that $|\nabla g|\leq K'$;
\item $f:[0,T]\times\IR^d\times\IR^m\times \IR^{m\times d} \to \IR^m$ is continuously differentiable with uniformly bounded derivatives, i.e. there exists a constant $K>0$ such that\footnote{\label{matrixtensorfootnote}We remark that this bound is taken over the corresponding Euclidean norm of the derivatives matrix/tensor. To avoid possible confusion when using tensors one can always interpret $f$ in the variable $z\in\IR^{m\times d}$ as taking not a matrix but a sequence of $d$-dimensional vectors $z_i\in \IR^d$ ($i\in\{1,\cdots,m\}$). The condition would then read $\sum_{i=1}^m |\nabla_{z_i} f| \leq \sqrt{K/3}$ where 
$f:[0,T]\times\IR^d\times\IR^m\times \underbrace{\IR^{d}\times \cdots \times \IR^d}_{m-\text{times}} \to \IR^m$.} $|\nabla_x f|,\, |\nabla_y f|,\,|\nabla_z f| \leq \sqrt{K/3}$ holds uniformly in all variables; $f$ satisfies a uniform Lipschitz condition with Lipschitz constant $\sqrt{K/3}$.
\item $b: [0,T] \times \IR^d \to \IR^d$ and $\sigma : [0,T] \times \IR^d \to \IR^{d\times d}$ are continuously differentiable functions with bounded derivatives; $|b(\cdot,0)|$ and $|\sigma(\cdot,0)|$ are uniformly bounded; $\sigma$ is elliptic; 
\item $\big( \int_0^T |f(s,0,0,0)|^2 \uds \big)^{p/2} < \infty$ for $p\geq 2$;
\item $f(t,\cdot,\cdot,\cdot)\1_{(-\infty,0)}(t) = 0$; 
\end{enumerate}
Condition (F3) is a standard assumption which guarantees the existence and uniqueness of the solution of SDE \eqref{eq:fwd1}. Furthermore, condition (F2) implies that the generator is uniformly Lipschitz continuous in $(x,y,z) \in \IR^d \times \IR^m \times \IR^{m\times d}$. In analogy to conditions (H2) and (H2') from section \ref{sec:delay_bsde}, let us write down the following implication of the Lipschitz condition (F2): with the constant $K>0$ chosen above, for any $t\in[0,T]$ and any sufficiently integrable vector or matrix valued processes $u,u'$, $y,y'$ and $z,z'$ it holds that 

\begin{align*}
\text{(F2')} \quad &\Big| f\big(t,(u\cdot\ax)(t),(y\cdot\ay)(t),(z\cdot\az)(t) \big) - f\big(t,(u'\cdot\ax)(t),(y'\cdot\ay)(t),(z'\cdot\az)(t) \big) \Big|^2
\\
&\qquad
\leq K\Big( \big| (u\cdot \ax) (t) - (u'\cdot \ax) (t) \big|^2
\\
&\hspace{2.6cm} + \big| (y\cdot \ay) (t) - (y'\cdot \ay) (t) \big|^2
 + \big| (z\cdot \az) (t) - (z'\cdot \az) (t) \big|^2 \Big)
\\
&\qquad
\leq K \ax([-T,0])  \big((x-x')^2\cdot\ax\big)(t)+ L\Big(\big((y-y')^2\cdot\ay\big)(t) + \big((z-z')^2\cdot\az\big)(t)\Big)
\end{align*}
where $L:=K \alpha$ with $\alpha$ defined in \eqref{eq:alpha1}. For a fixed $x \in\IR^d$, the existence and uniqueness of solutions to the backward equation \eqref{eq:bwd1} in $\cS^2_\beta \times \cH^2_\beta$ is guaranteed under the assumptions (F0)-(F5) together with the compatibility criterion from Theorem \ref{theo:DelongImkeller_thm2.1} on the terminal time and the Lipschitz constant $L=K\alpha$, i.e.
\begin{align*}
\big(8T + \frac1\beta \big) L \int_{-T}^0 e^{-\beta s} \rho(\uds) \max\{1,T\} < 1,\ \text{ for }\rho \in \{ \ay,\az\}.
\end{align*}
To extend the result to $\cS^p_\beta \times \cH^p_\beta$ for $p >2$, one only needs to replace the condition above by the compatibility condition from Theorem \ref{theo:picard},
\begin{align*} 
2^{p/2-1}C_p \Big(L T \int_{-T}^0 e^{-\beta s} \rho(\ud s) \Big)^{p/2} \max\{1,T^{p/2}\} < 1,\ \text{ for }\rho \in  \{\ay,\az \}.
\end{align*}
Throughout this section, given $p \geq 2$, we will assume that for every $x \in \IR^d$, the FBSDE \eqref{eq:fwd1}-\eqref{eq:bwd1} admits a unique solution $(X^x,Y^x,Z^x) \in \cS^q_\beta(\IR^d) \times \cS^p_\beta(\IR^m) \times \cH^p_\beta(\IR^{m \times d})$ for all $q\geq 2$.

\subsection{G\^ateaux and Norm differentiability}

In this section we investigate  the variational differentiability of the solution $(X^x,Y^x,Z^x)$ of the time delayed FBSDE \eqref{eq:fwd1}-\eqref{eq:bwd1} with respect to the Euclidean parameter $x\in\IR^d$, i.e. with respect to the initial condition of the forward diffusion. By a well known result (see e.g. \cite{Protter}), (F3) implies that the forward component $X^x$ is differentiable with respect to the parameter $x\in\IR^d$. It is natural to pose the question whether this smoothness is carried over to $(Y^x,Z^x)$ in the setting of FBSDE with time delayed generators.
In all this section we fix $h$ an element of $\IR^{d} \setminus \{0\}$. Our goal is to show that the variational equations of \eqref{eq:fwd1}-\eqref{eq:bwd1} are given by 
\begin{align}
\label{eq:nabla_X}
\nabla X^x_t h&= h + \int_0^t \nabla b(s,X^x_s) \nabla X^x_s h \; \uds + \int_0^t \nabla \sigma(s,X^x_s) \nabla X_s h \; \udws,\\
\label{eq:nabla_Y}
\nabla Y^x_t h &= \nabla g(X^x_T) \nabla X^x_T h \; - \int_t^T \nabla Z^x_s h \; \udws 
+ \int_t^T 
\big\langle (\nabla f)\big(s, \Theta^x(s)\big)
, (\nabla \Theta^x h)(s) \big\rangle \uds,
\end{align}
where the notation $\nabla X^x$ (respectively $\nabla Y^x$ and $\nabla Z^x$) denote the G\^ateaux derivatives of $X^x$ (respectively $Y^x$ and $Z^x$) in the direction $h$ and $(\nabla\Theta^x h)(t)$ is to be understood in the same fashion as in \eqref{eq:theta}, i.e. 
\begin{align}
\label{eq:nablatheta}
(\nabla\Theta^x h)(t)
&=\big((\nabla X^x h \cdot \ax)(t),(\nabla Y^x h \cdot \ay)(t),(\nabla Z^x h \cdot \az)(t)\big), \quad t\in[0,T].
\end{align}
Note that (F3) implies that \eqref{eq:nabla_X} admits a unique solution in $\cS^p_\beta$ for every $\beta  \geq 0$ and $p \geq 2$. Let $(X,Y,Z)$ and $\nabla X h$ solve \eqref{eq:fwd1}-\eqref{eq:bwd1} and \eqref{eq:nabla_X} respectively and let $\Theta^x$ be as defined by \eqref{eq:theta}. Now consider the BSDE with the linear time delayed generator for $t\in[0,T]$
\begin{align}
\label{eq:nabla_YBis}
P_t h &= \nabla g(X^x_T) \nabla X_T^x h - \int_t^T Q_s h \; \udws
+\int_t^T \widehat{F}\big(s,(P h \cdot \ay)(s),(Q h \cdot \az)(s)\big)\uds ,
\end{align}
where $\widehat{F}:\Omega \times [0,T] \times \IR^m\times \IR^{m \times d}\to \IR^m$, 
$\widehat{F}(t,p,q)=\langle (\nabla f)\big(t, \Theta^x(t)\big) , \big((\nabla X^x h \cdot \ax)(t),p,q\big) \rangle$.

The next corollary states, using Theorem \ref{theo:DelongImkeller_thm2.1} and Proposition \ref{prop:apriori_p2}, a result concerning the existence and uniqueness of solution to \eqref{eq:nabla_YBis}. This solution process will then serve as the natural candidate (in some sense) for $\nabla_x Y^xh$ and $\nabla_x Z^x h$, solution to \eqref{eq:nabla_Y}.
\begin{corollary}
\label{cor:nabla_Y_exists}
Let $p \geq 2$, $h\in \IR^{d} \setminus \{0\}$ and $\beta>0$. Assume that (F0)-(F5) are satisfied and let $L>0$ be as in (F2'). If $p>2$ assume that $T$, $K$, $\alpha$ are chosen like in Proposition \ref{lemma:apriori} and satisfy in addition 
\begin{align}
\label{section 3 compatibility condition}
2^{p/2-1}C_p \Big(L T \int_{-T}^0 e^{-\beta s} \rho(\ud s) \Big)^{p/2} \max\{1,T^{p/2}\} < 1,\  \text{ for }\rho\in\{\ay,\az\},
\end{align}
If $p=2$ assume $T$, $K$, $\alpha$ are chosen such that the conditions of Theorem \ref{theo:DelongImkeller_thm2.1} and of Proposition \ref{prop:apriori_p2} hold. Then for every fixed $x$ in $\mathbb{R}^d$, BSDE \eqref{eq:bwd1} has a unique solution $(Y,Z) \in \cS^p_\beta \times \cH^p_\beta$ and 
BSDE \eqref{eq:nabla_YBis} has a unique solution $(P h,Q h) \in \cS^p_\beta \times \cH^p_\beta$.
\end{corollary}
\begin{proof}
Given the known properties of $X$ and $\nabla X$ (and hence of $\nabla Xh$) it is easy to see that $\xi=\nabla g(X^x_T) \nabla X^x_T h$ and $\widehat{F}(\cdot,0,0)$ satisfy conditions (H1), (H3) and (H4). We recall Remark \ref{constants Di independent of data} to say that the several compatibility conditions \eqref{section 3 compatibility condition} as well as the conditions in Proposition \ref{lemma:apriori} depend only on the Lipschitz constant $K$ of (F2), the delay measures $\ay$, $\az$, $T$ and the dimension of the equations. 

From the definition of $\widehat{F}$ and using the bounds of the (spatial) derivatives of $f$ assumed in (F2) it is clear that $\widehat{F}$ satisfies a standard Lipschitz condition (in the spatial variables). In particular, take $p,p'\in \IR^m$ and\,\footnote{Or a sequence of $q_i,q_i'\in\IR^m$ with $i\in\{1,\cdots,d\}$ as we saw in page \pageref{matrixtensorfootnote}'s footnote.
} $q,q'\in\IR^{m\times d}$, then via Minkowski's and Cauchy-Schwarz inequalities along with (F2) we have
\begin{align*}
|\widehat{F}(t,p,q)-\widehat{F}(t,p',q')|
& 
\leq \big|\langle (\nabla_y f)\big(t, \Theta^x(t)\big) , (p-p') \rangle\big|+\big|
\langle (\nabla_z f)\big(t, \Theta^x(t)\big) , (q-q') \rangle \big|
\\
&
\leq |(\nabla_y f)|\,|p-p'| +|(\nabla_z f)|\,|q-q'|
\leq \sqrt{K/3}(\,|p-p'| + |q-q'|\,).
\end{align*}
And hence $\widehat{F}$ satisfies exactly the same Lipschitz condition as $f$. 
Furthermore, the delay measures appearing in $\widehat{F}$ are exactly the same ones as those that appear in $f$. We can thus conclude that the Lipschitz constant, the delay measures, terminal time $T$ and dimensions for $f$ and $\widehat{F}$ are the same. Under this corollary's assumptions, the conditions of 
Theorem \ref{theo:picard} are satisfied for both BSDE \eqref{eq:bwd1} and \eqref{eq:nabla_YBis}. The existence of a unique solution $(Y,Z)$ and $(Ph,Qh)$ in $\cS^p_\beta \times \cH^p_\beta$ of \eqref{eq:bwd1} and \eqref{eq:nabla_YBis} (respectively) follows from Theorem \ref{theo:picard} (and Theorem \ref{theo:DelongImkeller_thm2.1}).
\end{proof}
The solution of BSDE \eqref{eq:nabla_YBis} serves now as the natural candidate for the variational derivatives of $(Y,Z)$ solution of \eqref{eq:nabla_Y}. If one shows that $(\nabla Y^xh,\nabla Z^xh)$ exist in some sense then by the uniqueness of the solution of \eqref{eq:nabla_YBis}, the solutions to \eqref{eq:nabla_Y} and \eqref{eq:nabla_YBis} must coincide, i.e. $\big(\nabla Y^x h,\nabla Z^x h\big) = \big(P h,Q h\big)$ holds almost surely. 

For the rest of the section, we assume that all assumptions ensuring the existence and uniqueness of the variational equations \eqref{eq:nabla_X}-\eqref{eq:nabla_Y} are fulfilled, i.e. we assume that the assumptions of Corollary \ref{cor:nabla_Y_exists} hold. In our next result we show the mapping $x\mapsto (Y^x,Z^x)$ is differentiable in an adequate sense.

\begin{proposition}
\label{prop:diff}
Take $p\geq 2$ and assume the conditions of Corollary \eqref{cor:nabla_Y_exists} hold. Then for any $x\in\IR^d$ the solution $(X^x, Y^x, Z^x)$ of the FBSDE \eqref{eq:fwd1}-\eqref{eq:bwd1} is norm-differentiable in the following sense: 
\[
\lim_{\varepsilon \to 0} \left\| \frac{Y^{x+\varepsilon h}-Y^x}{\varepsilon} - \nabla Y^x h \right\|^p_{\cS^p_\beta}=\lim_{\varepsilon \to 0} \left\| \frac{Z^{x+\varepsilon h}-Z^x}{\varepsilon} - \nabla Z^x h\right\|^p_{\cH^p_\beta}=0, \quad \forall h \in \mathbb{R}^d \setminus\{0\},
\]
where $(\nabla Y^x h, \nabla Z^x h)$ is the unique solution of the BSDE
\begin{align*}
\nabla Y^x_t h &= \nabla g(X^x_T) \nabla X_T^x h- \int_t^T \nabla Z_s^x h \; \udws + \int_t^T 
\big\langle (\nabla f)\big(s, \Theta^x(s)\big) , (\nabla\Theta^x h)(s) \big\rangle \uds,
\end{align*}
with $\Theta^x$ and $\nabla\Theta^x$ defined by \eqref{eq:theta} and \eqref{eq:nablatheta} respectively.
\end{proposition}
\begin{proof}
Let $x \in \IR^d$,
$t\in[0,T]$ and $\varepsilon>0$. We use the following notations
\begin{align}
\label{eq:A}
A_{s,\mathcal{X}}&:=\int_0^1 \nabla_x f\Big(s,(X^{x}\cdot \ax)(s)+\theta \big((X^{x+\varepsilon h}-X^{x})\cdot \ax\big)(s),
\nonumber\\
&\hspace{3.0cm}(Y^{x+\varepsilon h}\cdot \ay)(s),(Z^{x+\varepsilon h}\cdot \az)(s)\Big) \ud \theta,
\nonumber\\
A_{s,\mathcal{Y}}&:=\int_0^1 \nabla_y f\Big(s,(X^{x}\cdot \ax)(s),
\nonumber\\
&\hspace{3.0cm}(Y^{x}\cdot \ay)(s)+\theta \big((Y^{x+\varepsilon h}-Y^{x})\cdot \ay\big)(s),(Z^{x+\varepsilon h}\cdot \az)(s)\Big) \ud \theta,
\\
A_{s,\mathcal{Z}}&:=\int_0^1 \nabla_z f\Big(s,(X^{x}\cdot \ax)(s),
\nonumber\\
&\hspace{3cm}(Y^{x}\cdot \ay)(s),(Z^{x}\cdot \az)(s)+\theta \big((Z^{x+\varepsilon h}-Z^{x})\cdot \az\big)(s)\Big) \ud \theta.
\nonumber
\end{align}
We remark that although the processes $A$ depends on $\varepsilon$ and $x$, for the sake of notational simplicity we do not write this dependence explicitly. We remark also that by assumption (F2) the processes $\vert A_{\cdot,\ast} \vert \leq \sqrt{K/3}$ for $\ast=\mathcal{X}, \mathcal{Y}, \mathcal{Z}$, in particular they are uniformly bounded in $x$ and $\varepsilon$.

We denote by $(P h,Q h$ the solution of the BSDE \eqref{eq:nabla_YBis} which coincides with $(\nabla Y h,\nabla Z h)$. We define the auxiliary processes
$\xi:=\big(g(X_T^{x+\varepsilon h})-g(X_T^x)\big)/{\varepsilon}-\nabla g(X_T^x) \nabla X_T^x h$, 
\begin{align}
\label{definition-of-X-tilde}
&U:=\frac{Y^{x+\varepsilon h}-Y^x}{\varepsilon}-P h,\ \ V:=\frac{Z^{x+\varepsilon h}-Z^x}{\varepsilon}-Q h,\ \text{ and }\ \tilde{X}:=\frac{X^{x+\varepsilon h}-X^x}{\varepsilon}-\nabla X^x h.
\end{align} 
Notice that from Assumption (F2) and the standard SDE theory we have that $\tilde{X}$ is well defined and $\tilde{X}\in\cS^p_\beta$ for any $b\geq 0$ and $p\geq 2$. We now claim and prove that 
\[
\lim_{\varepsilon \to 0} \| U \|^p_{\cS^p_\beta}=\lim_{\varepsilon \to 0} \| V \|^p_{\cH^p_\beta}=0,\quad \text{for arbitrary } x \in \IR^d.
\]
This result obviously proves the norm differentiability. To start with, we have 
\begin{align*}
U_t&=\xi+\int_t^T \frac{f(s,\Theta^{x+\varepsilon h}(s))-f(s,\Theta^{x}(s))}{\varepsilon} \uds \\
&- \int_t^T \big\langle (\nabla f)\big(s, \Theta^x(s)\big)
, \big( (\nabla X^x h \cdot \ax)(s), (P h \cdot \ay)(s), (Q h \cdot \az)(s)\big) \big\rangle \uds - \int_t^T V_s \udws.
\end{align*}
By construction the above equation is well defined, since for any $x$ and $\varepsilon$ all the involved processes are known a priori to exist and have the convenient integrability properties. The format of the above dynamics is still not convenient for our computations so we transform it into the more familiar dynamics of a delay BSDE. Using the identity $\phi(x)-\phi(y)=(x-y) \int_0^1 \nabla\phi(y+\theta(x-y)) \ud\theta$ for a continuously differentiable function $\phi:\mathbb{R}^a \to \mathbb{R}^b$ ($a$ and $b$ being arbitrary non-zero integers), the previous equation leads to
\begin{align}
\label{eq:diffBSDEtemp}
\nonumber
U_t&=\xi+\frac{1}{\varepsilon}\int_t^T\big[\, A_{s,\mathcal{X}} \big((X^{x+\varepsilon h}-X^x)\cdot \ax\big)(s)
\\
\nonumber
& \hspace{3cm} 
+ A_{s,\mathcal{Y}} \big((Y^{x+\varepsilon h}-Y^x)\cdot \ay\big)(s)+A_{s,\mathcal{Z}} \big((Z^{x+\varepsilon h}-Z^x\big)\cdot \az\big)(s) \,\big]\uds
\\
\nonumber
&
- \int_t^T \big\langle (\nabla f)\big(s, \Theta^x(s)\big)
, \big( (\nabla X^x h \cdot \ax)(s), (P h \cdot \ay)(s), (Q h \cdot \az)(s)\big) \big\rangle \uds - \int_t^T V_s \ud W_s 
\\ 
&
=\xi+\int_t^T \Phi\Big(s,(\tilde{X} \cdot \ax)(s),(U \cdot \ay)(s),(V \cdot \az)(s) \Big) \uds - \int_t^T V_s \ud W_s,
\end{align}
with $\tilde{X}$ given in \eqref{definition-of-X-tilde}, $ \Phi(t,x,y,z):=R_t+x A_{t,\mathcal{X}} + y A_{t,\mathcal{Y}} + z A_{t,\mathcal{Z}}$
and
\begin{align*}
R_t&:= -\big\langle (\nabla f)\big(t, \Theta^x(t)\big)
, \big( (\nabla X^x h \cdot \ax)(t), (P h \cdot \ay)(t), (Q h \cdot \az)(t)\big) \big\rangle\\
&\qquad \qquad + A_{t,\mathcal{X}} (\nabla X^x \cdot \ax)(t) + A_{t,\mathcal{Y}} (P h \cdot \ay)(t) + A_{t,\mathcal{Z}} (Q h \cdot \az)(t).
\end{align*} 
We now aim at using the results of Section 2 on the family (index by $\varepsilon$) of auxiliary delay BSDEs \eqref{eq:diffBSDEtemp}. In view of the uniform boundedness of the processes $A$ and the linearity of the driver $\Phi$, we can repeat the arguments used in the proof of Corollary \ref{cor:nabla_Y_exists} to conclude that under the assumptions of this proposition the data of BSDE \eqref{eq:diffBSDEtemp} (Lipschitz constant, delay measure and terminal time) satisfies uniformly in $\varepsilon$ the assumptions of Corollary \ref{cor:nabla_Y_exists} as well.

Applying the a priori estimate of Proposition \ref{lemma:apriori} or the moment estimate from Corollary \ref{coro:momentestimates} to the BSDE \eqref{eq:diffBSDEtemp} and taking into account that $\Phi$ satisfies (F2), we get 
\begin{align}
\label{eq:diffBSDEtemp2}
\| U \|^p_{\cS^p_\beta} +\| V \|^p_{\cH^p_\beta} &\leq C_p \Big\{ \IE \Big[ (e^{\beta T}|\xi|^2)^{p/2} \Big] + \IE \Big[ \big(\int_0^T e^{\beta s} \big|\Phi\big(s,(\tilde{X} \cdot \ax)(s),0,0\big)\big| \uds \big)^p \Big] \Big\} 
\nonumber 
\\
& 
\leq C \Big\{ \IE \Big[ (e^{\beta T}|\xi|^2)^{p/2} \Big] + \|  \tilde{X} \|^2_{\cH^p_\beta} + \IE \Big[ \big( \int_0^T e^{\beta s} |R_s| \uds \big)^p \Big] \Big\},
\end{align} 
for some constant $C>0$ (where we have used that $A_{\cdot,\mathcal{X}}$ is uniformly bounded). We proceed to compute the limit of each term on the right hand side of \eqref{eq:diffBSDEtemp2} as $\varepsilon$ goes to zero. 

\medskip
We first deal with the second term of the right hand side of \eqref{eq:diffBSDEtemp2}. Define 
$$ \hat{\sigma}_t:=\int_0^1 \nabla \sigma\big(t,X_t^{x}+\theta (X_t^{x+\varepsilon h}-X_t^{x})\big) \ud\theta \quad \textrm{ and } \quad \hat{b}_t:=\int_0^1 \nabla b\big(t,X_t^{x}+\theta (X_t^{x+\varepsilon h}-X_t^{x})\big) \ud\theta. $$
Note that $\tilde{X}\in \cS^p$ for any $p\geq 2$ (see \eqref{definition-of-X-tilde}) and solves the linear SDE
\begin{align}
\label{X-tilde-SDE}
 \tilde{X}_t=J_t+\int_0^t [\,\hat{\sigma}_s \tilde{X}_s\,] \ud W_s + \int_0^t\,[ \hat{b}_s \tilde{X}_s\,] \uds,
\end{align}
where $J$ is given by 
\[
J_t:=\int_0^t [\,\nabla X_s^x h (\hat{\sigma}_s-\nabla \sigma(s,X_s^x))\,] \ud W_s + \int_0^t[\, \nabla X_s^x h \big(\hat{b}_s-\nabla b(s,X_s^x)\big) \,]\uds.
\]
Given the known properties of $\nabla X$ and the fact that $\hat{b},\hat{\sigma},\nabla b$, and $\nabla \sigma$ are uniformly bounded we have that $J\in \cS_0^p$ for any $p\geq 2$. Indeed, Doob's inequality leads to
\begin{align*}
&\IE\Big[ \Big( \sup_{t\in [0,T]} \big\vert \int_0^t[\, \nabla X_s^x h\big(\hat{\sigma}_s-\nabla \sigma(s,X_s^x)\big)\,] \ud W_s \big\vert^2 \Big)^{p/2} \Big]
\leq C
\big\|\nabla X^x h \big(\hat{\sigma}-\nabla \sigma(\cdot,X^x)\big)\big\|_{\cH^p}^{p}
<\infty.
\end{align*}
Moreover, note that by Lebesgue's dominated convergence theorem 
\[
\lim_{\varepsilon \to 0}
 \|\nabla X^x h \big(\hat{\sigma}-\nabla \sigma(\cdot,X^x)\big)\|_{\cH^p}^{p}
 =0.\]
Similarly, using Jensen's inequality, the finite variation part of $J$ is an element of $\cS^p_0(\mathbb{R})$ and 
$$\lim_{\varepsilon \to 0} \|J\|_{\cS^p_0}=0.$$ 
Now we derive the following estimate for $\tilde{X}$ in terms of the norm of $J$ 
\begin{equation}
\label{eq:esttildeX}
\|\tilde{X} \|_{\cS^p_\beta}\leq C\, \E[\sup_{t\in [0,T]} |\tilde{X}_t|^p] \leq  C\, \| J \|_{\cS^p_0}
\end{equation} 
which will show that $\lim_{\varepsilon \to 0} \| \tilde{X} \|_{\cS^p_\beta}=0$.
Indeed equation \eqref{X-tilde-SDE} implies that:
\[
\E[\sup_{0\leq r\leq t} |\tilde{X}_r|^p] \leq C\, \IE\Big[ \sup_{0\leq r \leq t} |J_r|^p + \sup_{0\leq r\leq t} \big| \int_0^r [\,\hat{\sigma}_s \tilde{X}_s\,] \ud W_s \big|^p +  \sup_{0 \leq r \leq t} \big| \int_0^r \,[ \hat{b}_s \tilde{X}_s\,] \uds \big|^p \Big].
\]
Applying Burkholder-Davis-Gundy inequality to the second term in the right hand side, we get:
\[
\E[\sup_{0\leq r\leq t} |\tilde{X}_r|^p] \leq C\, \IE\Big[ \sup_{0\leq r \leq t} |J_r|^p + \big| \int_0^t |\hat{\sigma}_s \tilde{X}_s|^2 \uds \big|^{p/2} +  \sup_{0 \leq r \leq t} \big| \int_0^r \,[ \hat{b}_s \tilde{X}_s\,] \uds \big|^p \Big].
\]
Jensen's inequality and the fact that $\hat{\sigma}$ and $\hat{b}$ are bounded imply that:
\[
\E\big[\sup_{0\leq r\leq t} |\tilde{X}_r|^p\big] \leq C \, \E\Big[ \sup_{0\leq r \leq t} |J_r|^p +\int_0^t |\tilde{X}_s|^p \uds \Big]
\]
hence
\[
\E\big[\sup_{0\leq r\leq t} |\tilde{X}_r|^p\big] \leq C\, \Big\{ \E[\sup_{0\leq r \leq t} |J_r|^p] + \int_0^t \E[\sup_{0\leq r \leq s} |\tilde{X}_r|^p] \uds \Big\}.
\]
Gronwall's lemma finally entails estimate \eqref{eq:esttildeX} and  
thus $\lim_{\varepsilon \to 0} \| \tilde{X} \|_{\cS^p_\beta}=0$.
\vspace{0.3cm}

Let us consider the terminal condition term in \eqref{eq:diffBSDEtemp2}. Denoting 
$$
\hat{g}:=\int_0^1 \nabla g\big(X_T^{x}+\theta (X_T^{x+\varepsilon h}-X_T^{x})\big) \ud\theta,
$$
it holds that
\begin{align*}
\IE \Big[ (e^{\beta T}|\xi|^2)^{p/2} \Big]
&
=
e^{\beta T p/2}\big\|\hat{g} \big(\frac{X^{x+\varepsilon h}_T-X^x_T}{\varepsilon}-\nabla X_T^x h \big) + \big(\hat{g}-\nabla g(X_T^x)\big) \nabla X_T^x h\big\|_{L^p}^p
\\
&
\leq C\Big\{
\big\|\frac{X^{x+\varepsilon h}_T-X^x_T}{\varepsilon}-\nabla X_T^x h \big\|_{L^p}^p
+
\big\|\,|\nabla X_T^x h \vert\, \vert \hat{g}-\nabla g(X_T^x)|\,\big\|_{L^p}^p
\\
&
\leq C \Big\{
\|\tilde{X}\|_{\cS^p_0}^p
+
\big\|\,|\nabla X_T^x h \vert\, \vert \hat{g}-\nabla g(X_T^x)|\,\big\|_{L^p}^p\Big\}
\Big\}
 \underset{\varepsilon \to 0}{\longrightarrow} 0,
\end{align*}
where we have used Lebesgue's dominated convergence theorem for the second summand and the estimate obtained above on the norm of $\tilde{X}$ for the first one.

\vspace{0.3cm}
Now, let us consider the last term on the right hand side of \eqref{eq:diffBSDEtemp2}. We have that 
\begin{align*}
\IE \Big[ \Big( \int_0^T e^{\beta s} |R_s| \uds \Big)^p \Big]
&\leq C\, \IE \left[ \Big( \int_0^T e^{\beta s} \left|\left(A_{s,\mathcal{X}}-\nabla_x f\big(s, \Theta^x(s)\big)\right) (\nabla X^x h \cdot \ax)(s)\right| \uds \Big)^p \right]\\
&\qquad+ C\, \IE \left[ \Big( \int_0^T e^{\beta s} \left|\left(A_{s,\mathcal{Y}}-\nabla_y f\big(s, \Theta^x(s)\big)\right) ( P h \cdot \ay)(s)\right| \uds \Big)^p \right]\\
&\qquad+ C\, \IE \left[ \Big( \int_0^T e^{\beta s} \left|\left(A_{s,\mathcal{Z}}-\nabla_z f\big(s, \Theta^x(s)\big)\right) ( Q h \cdot \az)(s)\right| \uds \Big)^p \right].
\end{align*}
Standard arguments yield (note that $\varepsilon>0$ is implicitly contained in $A_{t,\mathcal{X}}$, see \eqref{eq:A})
\begin{align*}
A_{t,\mathcal{X}} \longrightarrow \nabla_x f\big(t, \Theta^x(t)\big) \quad \text{as $\varepsilon \to 0$ in probability, for }
\udt\text{-a.a. } t \in [0,T].
\end{align*}
Moreover, Proposition \ref{lemma:apriori} and the previous calculations show that
\begin{align*}
&\| Y^{x+\varepsilon h} -Y^x \|^p_{\cS^p_\beta} + \| Z^{x+\varepsilon h} -Z^x \|^p_{\cH^p_\beta} 
\\
&\hspace{1cm}
\leq C \, \big\{ 
e^{ \beta T\,p}\|g(X^{x+\varepsilon h}) - g(X^{x}) \|^p_{L^p}
+ 
\|X^{x+\varepsilon h}-X^x \|^p_{\cH^p_\beta} \big\}\underset{\varepsilon \to 0}{\longrightarrow} 0,
\end{align*}
for some positive constant $C$. This implies for $\udt$-a.a. $t\in[0,T]$
\begin{align*}
Y_t^{x+\varepsilon h} \to Y_t^x, \quad Z_t^{x+\varepsilon h} \to Z_t^x, \quad \text{as $\varepsilon\to 0$ in probability.}
\end{align*}
Since $\nabla_y f$, $\nabla_z f$ are continuous, it follows that for $\udt$-a.a. $t\in[0,T]$
\begin{align*}
A_{t,\mathcal{Y}} \longrightarrow \nabla_y f\big(t, \Theta^x(t)\big), \quad \text{as $\varepsilon\to 0$ in probability,}\\
A_{t,\mathcal{Z}} \longrightarrow \nabla_z f\big(t, \Theta^x(t)\big), \quad \text{as $\varepsilon\to 0$ in probability.}
\end{align*} 
Thus, using Lemma \ref{lemma:interchange} and the fact that $P$ and $Q$ are square integrable, Lebesgue's dominated convergence theorem (which also holds, if almost sure convergence is replaced by convergence in probability, \textit{cf.} \cite{Shiryaev}, remark on p. 258) yields $\lim_{\varepsilon \to 0} \IE \big[ \big( \int_0^T e^{\beta s} |R_s| \uds \big)^p \big]=0$. Now  \eqref{eq:diffBSDEtemp2} yields that
\begin{align*}
\lim_{\varepsilon \to 0} \big\{ \| U \|^p_{\cS^p_\beta} +\| V \|^p_{\cH^p_\beta} \big\} = 0,
\end{align*}
which proves the claim.
\end{proof}

\subsection{Strong differentiability}

All previous assumptions on existence and uniqueness remain in force. In this section, we concentrate on the smoothness properties of the paths associated to the processes $(Y^x,Z^x)$. We assume throughout this section that $m=1$, i.e. the delay BSDE are now one-dimensional. A first result is obtained in the following

\begin{proposition}
\label{prop:cont}
Set $m=1$ and under the assumptions of Corollary \ref{cor:nabla_Y_exists} 
 we have for $x,x'\in\IR^d$ 
$$ \IE\big[ \sup_{0\leq t \leq T} |X_t^x-X_t^{x'}|^q  \big] \leq C |x-x'|^{q},\quad \text{for any }\ q\geq 2,$$
and for any $p>2$
$$ \IE\Big[ \sup_{0\leq t \leq T} \big( e^{\beta t} |Y_t^{x}-Y_t^{x'}|^2 \big)^{p/2} \Big] + \IE\Big[ \big( \int_0^T e^{\beta s} |Z_s^{x}-Z_s^{x'}|^2 \uds \big)^{p/2} \Big] \leq C |x-x'|^p.$$
Thus for every $x \in \IR^d$,
\begin{itemize}
\item the mapping $x\mapsto Y^x$ from $\mathbb{R}^d$ to the space of c\`adl\`ag functions equipped with the topology given by the uniform convergence on compacts sets is continuous $\IP$-almost surely,
\item the mapping $x\mapsto Z^x$ is continuous from $\mathbb{R}^d$ to $L^2([0,T])$ $\IP$-almost surely. 
\end{itemize}
In particular, for every $x \in \IR^d$,
\begin{itemize}
\item the mapping $x\mapsto Y_t^x$ from $\mathbb{R}^d$ to $\mathbb{R}$ is continuous for all $t \in [0,T]$, $\IP$-almost surely,
\item the mapping $x\mapsto Z^x_t(\omega)$ is continuous for every $x \in \IR^d$ and $ \udt \otimes \ud \IP$-almost all $(t,\omega)$.
\end{itemize}
\end{proposition}
\begin{proof}
The estimate on the forward process is classical (see \textit{e.g.} \cite[Theorem V.37 Equation (***) p.~309]{Protter}). In this proof, $C>0$ denotes a generic constant which may differ from line to line. We apply the a priori estimate from Proposition \ref{lemma:apriori} and get
\begin{align*}
&\IE\Big[ \sup_{0\leq t \leq T} \big( e^{\beta t} |Y_t^{x}-Y_t^{x'}|^2 \big)^{p/2} \Big] + \IE\Big[ \big( \int_0^T e^{\beta s} |Z_s^{x}-Z_s^{x'}|^2 \uds \big)^{p/2} \Big]\\
& \quad \leq C_p \Big\{ \IE\Big[ \big(e^{\beta T} |g(X_T^x)-g(X_T^{x'})|^2 \big)^{p/2} \Big]\\
&\quad\quad  + \IE\Big[ \big(\int_0^T e^{\frac{\beta}{2} s} |f\big(s,(X^x \cdot \ax)(s),\zeta(s)\big)-f\big(s,(X^{x'} \cdot \ax)(s),\zeta(s)\big)| \uds \big)^{p} \Big] \Big\}\\
& \quad \leq C \Big\{ \IE\Big[ \big(e^{\beta T} |g(X_T^x)-g(X_T^{x'})|^2 \big)^{p/2} \Big]\\
&\quad\quad  + \IE\Big[ \big(\int_0^T e^{\beta  s} |f(s,(X^x \cdot \ax)(s),\zeta(s))-f(s,(X^{x'} \cdot \ax)(s),\zeta(s))|^2 \uds \big)^{p/2} \Big] \Big\},
\end{align*}
with $ \zeta(\cdot):=\big((Y^{x'} \cdot \ay)(\cdot),(Z^{x'} \cdot \az)(\cdot)\big)$.
Using the mean value theorem and the boundedness of $\nabla f$ and $\nabla g$ (i.e. the Lipschitz property of $f$ and $g$), we deduce 
\begin{align*}
&\IE\Big[ \sup_{0\leq t \leq T} \big( e^{\beta t} |Y_t^{x}-Y_t^{x'}|^2 \big)^{p/2} \Big] + \IE\Big[ \big( \int_0^T e^{\beta s} |Z_s^{x}-Z_s^{x'}|^2 \uds \big)^{p/2} \Big]\\
& \quad \leq C \Big\{ \IE\Big[ \big(e^{\beta T} |X_T^x-X_T^{x'}|^2 \big)^{p/2} \Big] + \IE\Big[ \big(\int_0^T e^{\beta s} |((X^x-X^{x'}) \cdot \ax)(s)|^2 \uds \big)^{p/2} \Big] \Big\}
\\
& \quad
\leq C \Big\{ \IE\Big[ \big(e^{\beta T} |X_T^x-X_T^{x'}|^2 \big)^{p/2} \Big] + \IE\Big[ \big(\int_0^T e^{\beta s} |X_s^x-X_s^{x'}|^2 \uds \big)^{p/2} \Big] \Big\}
\\
&\quad\leq C |x-x'|^p,
\end{align*}
where the last two lines follow by applying the change of integration from \eqref{eq:tmp_02} and the first claim of the proposition. The continuity properties of the mappings $x \mapsto Y^x$ and $x \mapsto Z^x$ are now obtained by an application of Kolmogorov's continuity criterion (see for example \cite[IV.7 Corollary 1]{Protter}).
\end{proof}
If the generator exhibits additional regularity, it even turns out that the paths of $x \mapsto Y^x$ are continuously differentiable.

\begin{theorem}
\label{strong-diff-theorem}
Let $\beta>0$ and assume the conditions of Proposition \ref{prop:diff} can be verified for some $\widehat{p}>4$. Assume moreover that all (spatial) second order partial derivatives of $b,\sigma,g$ and $f$ exist, are continuous and uniformly bounded. Then, for any  $(x,\varepsilon), (x',\varepsilon') \in \mathbb{R}^d \times (0,\infty)$, $h\in\IR^d$ and $p\in (2, {\widehat{p}}/{2}]$ it holds that 
\begin{align*}
\IE\Big[\sup_{0\leq t \leq T} \Big(e^{\beta t}
\Big|\frac{Y_t^{x+\varepsilon h}-Y_t^x}{\varepsilon}
-\frac{Y_t^{x'+\varepsilon' h}-Y_t^{x'}}{\varepsilon'}\Big|^2
\Big)^{p/2}\Big] \leq C\, \big(|x-x'|^2 + |\varepsilon-\varepsilon'|^2\big)^{p/2}.
\end{align*}
Thus $\nabla_x Y^x$ belongs to $\cH_\beta^{\widehat{p}}$ and the mapping $x\mapsto Y^x_t(\omega)$ is continuously differentiable for all $t\in[0,T]$, $\IP$-almost surely.
\end{theorem}
It is known that the existence of the partial derivatives (or even all of the directional derivatives) of a function does not guarantee that the function is differentiable at a point. But it is if all the partial derivatives of the function exist and are continuous in a neighborhood of the point, then the function must be differentiable at that point and is in fact of class $C^1$.

Under the assumption that $m=1$ and the subsequent corollary of the Theorem in the previous section, we know that the all (spatial) partial derivatives of $Y^x$ exist. The main result of Theorem \ref{strong-diff-theorem} is the continuity of those partial derivatives.
\begin{proof}
As in the previous proof, $C>0$ denotes a generic constant which can differ from line to line. Let $p>2$, $t\in[0,T]$ and $h\in \mathbb{R}^d\setminus\{0\}$. 
For $(x,\varepsilon) \in \mathbb{R}^d \times (0,\infty)$ let $U^{x,\varepsilon} := \frac{Y^{x+\varepsilon h} - Y^x}{\varepsilon}$, $V^{x,\varepsilon} := \frac{Z^{x+\varepsilon h} - Z^x}{\varepsilon}$, $\xi^{x,\varepsilon} := \frac{g(X_T^{x+\varepsilon h}) - g(X_T^x)}{\varepsilon}$ and $\tilde{X}^{x,\varepsilon} := \frac{X^{x+\varepsilon h}-X^x}{\varepsilon}$. Using the notation from the proof of Proposition \ref{prop:diff}, the pair $(U^{x,\varepsilon},V^{x,\varepsilon})$ satisfies the BSDE
$$ U_t^{x,\varepsilon}=\xi^{x,\varepsilon}+\int_t^T \Phi(s,\zeta^{x,\varepsilon}(s)) \uds - \int_t^T V_s^{x,\varepsilon}\udws,$$
with
$\zeta^{x,\varepsilon}(t):=\big((U^{x,\varepsilon} \cdot \ay)(t),(V^{x,\varepsilon} \cdot \az)(t)\big)$
and $\Phi(t,y,z):=(\tilde{X}^{x,\varepsilon} \cdot \ax)(t) A_{t,\mathcal{X}}^{x,\varepsilon} + y A_{t,\mathcal{Y}}^{x,\varepsilon} + z A_{t,\mathcal{Z}}^{x,\varepsilon}$. 
Note that the terms $A_{\cdot,\ast}^{x,\varepsilon}$ with $\ast = \mathcal{X},\mathcal{Y}, \mathcal{Z}$ are given by \eqref{eq:A}.  

For whatever choice of $(x,\varepsilon)$ we emphasize that the arguments used in the proof of Corollary \ref{cor:nabla_Y_exists} and Proposition \ref{prop:diff} hold true for the above auxiliary BSDE in what the applicability of the a priori estimate of Proposition \ref{lemma:apriori} is concerned.

Let another pair $(x',\varepsilon') \in \IR^d\times(0,\infty)$ be given. Applying Proposition \ref{lemma:apriori} yields
\begin{align*}
\| U^{x,\varepsilon}-U^{x',\varepsilon'} \|_{\cS^p_\beta}^p
&
\leq C_p \Big\{ \IE\Big[ \big(e^{\beta T} |\xi^{x,\varepsilon}-\xi^{x',\varepsilon'}|^2 \big)^{p/2} \Big] + \IE\Big[ \big( \int_0^T e^{\frac{\beta}{2} s} |\delta_2 \Phi(s)| \uds \big)^p \Big] \Big\},
\end{align*}
with 
\begin{align*}
&\delta_2 \Phi(t):= (\tilde{X}^{x,\varepsilon} \cdot \ax)(t) A_{t,\mathcal{X}}^{x,\varepsilon}-(\tilde{X}^{x',\varepsilon'} \cdot \ax)(t) A_{t,\mathcal{X}}^{x',\varepsilon'}
\\
&\hspace{2cm}
+(U^{x',\varepsilon'} \cdot \ay)(t) (A_{t,\mathcal{Y}}^{x,\varepsilon}-A_{t,\mathcal{Y}}^{x',\varepsilon'})+(V^{x',\varepsilon'} \cdot  \az)(t)(A_{t,\mathcal{Z}}^{x,\varepsilon}-A_{t,\mathcal{Z}}^{x',\varepsilon'}).
\end{align*}
Using the hypotheses on $f$ (\textit{i.e.} all partial derivatives up to order two are bounded), we find
\begin{align*}
|\delta_2 \Phi(t)|
&
\leq C \Big\{ |((\tilde{X}^{x,\varepsilon}-\tilde{X}^{x',\varepsilon'}) \cdot \ax)(t)| |A_{t,\mathcal{X}}^{x,\varepsilon}| + |(\tilde{X}^{x',\varepsilon'} \cdot \ax)(t)| |A_{t,\mathcal{X}}^{x,\varepsilon}-A_{t,\mathcal{X}}^{x',\varepsilon'}|
\\
&\quad 
+ |(U^{x',\varepsilon'} \cdot \ay)(s)||A_{t,\mathcal{Y}}^{x,\varepsilon}-A_{t,\mathcal{Y}}^{x',\varepsilon'}|+|(V^{x',\varepsilon'} \cdot \az)(t)| |A_{t,\mathcal{Z}}^{x,\varepsilon}-A_{t,\mathcal{Z}}^{x',\varepsilon'}|\Big\}.
\end{align*}
As a consequence 
\begin{align*}
&
\| U^{x,\varepsilon}-U^{x',\varepsilon'} \|_{\cS^p_\beta}^p
\\
&
\quad 
\leq C \Big\{ 
\|\xi^{x,\varepsilon}-\xi^{x',\varepsilon'}\|_{L^p}^{p}
+
\IE\Big[ \big( \int_0^T e^{\frac{\beta}{2} s} |((\tilde{X}^{x,\varepsilon}-\tilde{X}^{x',\varepsilon'}) \cdot \ax)(s)| |A_{s,\mathcal{X}}^{x,\varepsilon}| \uds \big)^p \Big]
\\
&\quad \quad 
+ 
\IE\Big[ \big( \int_0^T e^{\frac{\beta}{2} s} |(\tilde{X}^{x',\varepsilon'} \cdot \ax)(s)| |A_{s,\mathcal{X}}^{x,\varepsilon}-A_{s,\mathcal{X}}^{x',\varepsilon'}| \uds \big)^p \Big]
\\
&\quad \quad 
+ 
\IE\Big[ \big( \int_0^T e^{\frac{\beta}{2} s} |(U^{x',\varepsilon'} \cdot \ay)(s)||A_{s,\mathcal{Y}}^{x,\varepsilon}-A_{s,\mathcal{Y}}^{x',\varepsilon'}| \uds \big)^p \Big] 
\\
&\quad \quad
+ 
\IE\Big[ \big( \int_0^T e^{\frac{\beta}{2} s} |(V^{x',\varepsilon'} \cdot \az)(s)||A_{s,\mathcal{Z}}^{x,\varepsilon}-A_{s,\mathcal{Z}}^{x',\varepsilon'}| \uds \big)^p \Big] \Big\}
\\
&\quad 
\leq C \Big\{ 
\|\xi^{x,\varepsilon}-\xi^{x',\varepsilon'}\|_{L^p}^{p}
+ 
\| \tilde{X}^{x,\varepsilon}-\tilde{X}^{x',\varepsilon'} \|_{\cH_\beta^{2p}}^{p}
\| A_{\cdot,\mathcal{X}}^{x,\varepsilon} \|_{\cH_\beta^{2p}}^{p}
+
\| \tilde{X}^{x',\varepsilon'} \|_{\cH_\beta^{2p}}^{p}
\| A_{\cdot,\mathcal{X}}^{x,\varepsilon} - A_{\cdot,\mathcal{X}}^{x',\varepsilon'} \|_{\cH_\beta^{2p}}^{p}
\\
&
\qquad\qquad 
+ 
\| U^{x',\varepsilon'} \|_{\cH_\beta^{2p}}^{p}
\| A_{\cdot,\mathcal{Y}}^{x,\varepsilon} - A_{\cdot,\mathcal{Y}}^{x',\varepsilon'} \|_{\cH_\beta^{2p}}^{p}
+ 
\| V^{x',\varepsilon'} \|_{\cH_\beta^{2p}}^{p}
\| A_{\cdot,\mathcal{Z}}^{x,\varepsilon} - A_{\cdot,\mathcal{Z}}^{x',\varepsilon'} \|_{\cH_\beta^{2p}}^{p}
\Big\},
\end{align*}
where for each term we used the Cauchy-Schwarz inequality twice, that $e^{\frac{\beta}{2} t} \leq e^{\beta t}$ and \eqref{eq:tmp_02}. 
Since $(U^{x',\varepsilon'},V^{x',\varepsilon'})$ is a solution in $\cS^p_\beta \times \cH^p_\beta$ of a BSDE, it follows from Corollary \ref{coro:momentestimates} that the quantities $\IE\big[ \big( \int_0^T e^{\beta s} |U_s^{x',\varepsilon'}|^2 \uds \big)^p \big]$ and $\IE\big[ \big( \int_0^T e^{\beta s} |V_s^{x',\varepsilon'}|^2 \uds \big)^p \big]$ are finite and uniformly bounded in $\varepsilon'$. By the assumptions on $b$ and $\sigma$, we have
$$\IE\Big[ \big( \int_0^T e^{\beta s} |\tilde{X}_s^{x',\varepsilon'}|^2 \uds \big)^p \Big]^{1/2}<\infty.$$ 
In addition, by the boundedness of $\nabla f$ we have that $|A_{\cdot,\ast}^{x,\varepsilon}|$ and $|A_{\cdot,\ast}^{x',\varepsilon'}|$ are uniformly bounded (in their several parameters) with $\ast = \mathcal{X},\mathcal{Y},\mathcal{Z}$. Thus the estimate reduces to 
\begin{align}
\label{eq:strongdiff2}
\nonumber
\| U^{x,\varepsilon}-U^{x',\varepsilon'} \|_{\cS^p_\beta}^p
&\leq C \Big\{ 
\|\xi^{x,\varepsilon}-\xi^{x',\varepsilon'}\|_{L^p}^{p}
+
\| \tilde{X}^{x,\varepsilon}-\tilde{X}^{x',\varepsilon'} \|_{\cH_\beta^{2p}}^{p}
+
\| A_{\cdot,\mathcal{X}}^{x,\varepsilon} - A_{\cdot,\mathcal{X}}^{x',\varepsilon'} \|_{\cH_\beta^{2p}}^{p}
\nonumber
\\
& \hspace{1cm}
+
\| A_{\cdot,\mathcal{Y}}^{x,\varepsilon} - A_{\cdot,\mathcal{Y}}^{x',\varepsilon'} \|_{\cH_\beta^{2p}}^{p}
+
\| A_{\cdot,\mathcal{Z}}^{x,\varepsilon} - A_{\cdot,\mathcal{Z}}^{x',\varepsilon'} \|_{\cH_\beta^{2p}}^{p}
\Big\}.
\end{align}
Using the mean value theorem and the fact that the second order partial derivatives are bounded it holds that
\begin{align*}
&|A_{t,\mathcal{X}}^{x,\varepsilon}-A_{t,\mathcal{X}}^{x',\varepsilon'}| + |A_{t,\mathcal{Y}}^{x,\varepsilon}-A_{t,\mathcal{Y}}^{x',\varepsilon'}| + |A_{t,\mathcal{Z}}^{x,\varepsilon}-A_{t,\mathcal{Z}}^{x',\varepsilon'}|
\\
& \hspace{0.5cm}
\leq C \Big\{ \big(|X^{x+\varepsilon h}-X^{x'+\varepsilon' h}| \cdot \ax\big)(t) + \big(|Y^{x+\varepsilon h}-Y^{x'+\varepsilon' h}| \cdot \ay\big)(t) 
\\
& \hspace{1cm}
+\big(|Z^{x+\varepsilon h}-Z^{x'+\varepsilon' h}| \cdot \az\big)(t)  + \big(|X^{x}-X^{x'}| \cdot \ax\big)(t)
\\
& \hspace{1cm}
 + \big(|Y^{x}-Y^{x'}| \cdot \ay\big)(t) + \big(|Z^{x}-Z^{x'}| \cdot \az\big)(t)\Big\}.
\end{align*}
Plugging the right hand side of this inequality in \eqref{eq:strongdiff2} and using Lemma \ref{lemma:interchange} one gets
\begin{align*}
\| U^{x,\varepsilon}-U^{x',\varepsilon'} \|_{\cS^p_\beta}^p
&
\leq C \Big\{ 
\|\xi^{x,\varepsilon}-\xi^{x',\varepsilon'}\|_{L^p}^{p}
+
\| \tilde{X}^{x,\varepsilon}-\tilde{X}^{x',\varepsilon'} \|_{\cH_\beta^{2p}}^{p}
+
\| X^x-X^{x'} \|_{\cH_\beta^{2p}}^{p}
\\
&\qquad  
+
\| X^{x+\varepsilon h}-X^{x'+\varepsilon' h} \|_{\cH_\beta^{2p}}^{p}
+ 
\| Y^{x+\varepsilon h}-Y^{x'+\varepsilon' h} \|_{\cH_\beta^{2p}}^{p}
\\
&\qquad 
+ 
\| Z^{x+\varepsilon h}-Z^{x'+\varepsilon' h} \|_{\cH_\beta^{2p}}^{p}
+ 
\| Y^x-Y^{x'} \|_{\cH_\beta^{2p}}^{p}
+ 
\| Z^x-Z^{x'} \|_{\cH_\beta^{2p}}^{p}
\Big\}.
\end{align*} 
Since $b$, $\sigma$ and $g$ are twice continuously differentiable with bounded derivatives we have the following estimate 
$$ \IE\big[ \,|\xi^{x,\varepsilon}-\xi^{x',\varepsilon'}|^p \big] \leq C (|x-x'|^2 + |\varepsilon-\varepsilon'|^2)^{p/2},$$
which is proved for example in \cite[Lemma 7.4]{AnkirchnerImkellerDosReis}. This result combined with Proposition \ref{prop:cont} leads to
$$\IE\Big[\sup_{0\leq t \leq T} \big( e^{\beta t} |U_t^{x,\varepsilon}-U_t^{x',\varepsilon'}|^2\big)^{p/2}\Big] \leq C \big(|x-x'|^2 + |\varepsilon-\varepsilon'|^2\big)^{p/2}.$$  
The last claim of the theorem follows using Kolmogorov's continuity criterion (see for example \cite[IV.7 Corollary 1]{Protter}).
\end{proof}

\section{Representation formulas and path regularity}
\label{section:representation}

One of the fundamental results in the setting of FBSDE concerns the relationship between the Malliavin and the variational (classical) derivatives of the solution process: the Malliavin derivative of the solution of the BSDE can be expressed as a product of the BSDE's solution variational derivatives (with respect to the initial parameter of the SDE) and the variational derivatives of the forward diffusion. This relationship is known to hold both in the standard Lipschitz generator setting (see Proposition 5.9 of \cite{97KPQ}) as well as the quadratic generator case (see  e.g. Theorem 2.9 of \cite{ImkellerDosReis}) for classical BSDE without time delayed generators. 
\smallskip

In this section we show that this relationship still holds for decoupled FBSDE with time delayed generators. Such a result is somewhat surprising since it is normally dependent on a Markovian structure for the solution of the BSDE that exists for non-time delayed BSDE and which fails to materialize for time delayed BSDE. Imperative for this relationship to hold is the fact that the forward process $X$ is Markovian along with a good behavior of the terminal condition.
\smallskip

As in the previous section, whenever we consider the delay FBSDE \eqref{eq:fwd1}-\eqref{eq:bwd1}, we assume that all conditions to ensure the existence of a unique solution $(X,Y,Z)$ are in force. Moreover, since for $\beta \geq 0$, all $\beta$-norms are equivalent, in the following we content ourselves with giving results for $\beta=0$. Recall that we assume $m=1$, i.e. the delay BSDE is \emph{not} vector-valued.

\subsection*{Malliavin's differentiability of FBSDE with time delayed generators}

We recall Theorem 4.1 of \cite{DelongImkeller2}, modified to our the FBSDE setting. Theorem 4.1 from \cite{DelongImkeller2} shows that the solutions of time delayed BSDE are Malliavin differentiable, and as a consequence, it can be deduced that the solution of the time delayed FBSDE  \eqref{eq:fwd1}-\eqref{eq:bwd1} is also Malliavin differentiable. Under the condition (F3) on the coefficients of the forward equation \eqref{eq:fwd1}, the Malliavin differentiability of the forward process $X$ is a standard result, see for instance Theorem 2.2.1 in \cite{nualart1995}. We denote the solution to the equations \eqref{eq:fwd1}-\eqref{eq:bwd1} by $(X,Y,Z)$. The next result states the Malliavin differentiability of $(X,Y,Z)$. Using the notation introduced in Section 3, we define for $0\leq u\leq t\leq T$
\begin{align}
\label{eq:malliaviontheta}
\nonumber
(D_u \Theta)(t)
&=\big((D_u X \cdot \ax)(t),(D_u Y\cdot \ay)(t),(D_u Z\cdot \az)(t)\big) 
\\
&=\Big( \int_{-T}^0 D_u X_{t+v} \ax(\ud v), \int_{-T}^0 D_u Y_{t+v} \ay(\ud v), \int_{-T}^0 D_u Z_{t+v} \az(\ud v)  \Big).
\end{align}
We define in the canonical way\footnote{See Section 2.2 of \cite{ImkellerDosReis}, Section 5.2 of \cite{97KPQ} or simply \cite{nualart1995}} the space $\IL_{1,2}$ as the space of progressively measurable processes, $X\in\cH^2$, that are Malliavin differentiable and normed by $\|X\|_{\IL_{1,2}}=\IE[ \int_0^T |X_s|^2 \uds +\int_0^T \int_0^T |D_u X_s|^2 \uds \ud u]^{1/2}$.

\begin{theorem}
\label{malliavindifftheo}
Take $p=2$, $m=1$ and assume the conditions of Corollary \ref{cor:nabla_Y_exists} hold. Then $(X,Y,Z)$ are Malliavin differentiable and their derivatives $(DX,DY,DZ)$ solve uniquely in $\IL_{1,2}\times \IL_{1,2} \times\IL_{1,2}$ the following time delayed FBSDE: 
\begin{align}
\label{eq:DX}
D_u X_t &= \sigma(u,X_u) + \int_u^t \nabla_x b(s,X_s)D_u X_s \uds + \int_u^t \nabla_x \sigma(s,X_s) D_u X_s \udws,\\
\label{eq:malliavin_Y}
D_u Y_t &= \nabla g(X_T)D_u X_T - \int_t^T D_u Z_s \udws + \int_t^T 
\big\langle (\nabla f)\big(s, \Theta(s)\big) , (D_u \Theta)(s) \big\rangle \uds,
\end{align}
for $0 \leq u \leq t \leq T$ (zero otherwise) 
with $\Theta$ and $D\Theta$ given by \eqref{eq:theta} and \eqref{eq:malliaviontheta} respectively. Furthermore, $\{D_t Y_t: t\in[0,T]\}$ is a version of $\{Z_t:t\in[0,T]\}$.
\end{theorem}
\begin{proof}
The results concerning the forward component are well known, see \cite{nualart1995} or \cite{ImkellerDosReis}. The conditions of Corollary \ref{cor:nabla_Y_exists} ensure that Theorem 4.1 from \cite{DelongImkeller2} can be applied. Hence $Y$ and $Z$ are Malliavin differentiable. The representation of $Z$ by the trace of of the Malliavin derivative of $Y$ follows as well from the cited result.
\end{proof}

\subsection*{The representation formulas}

We now present the representation formulas for \eqref{eq:DX} and  \eqref{eq:malliavin_Y} which are effectively expressed in terms of the variational $\nabla X, \nabla Y$ and $\nabla Z$.
\begin{theorem}\label{proporepresentationformulas}
Let the conditions of Theorem 4.1 hold. Let $(X,Y,Z)$, $(\nabla X, \nabla Y, \nabla Z)$ and $(D X, D Y, D Z)$ denote the solutions of FBSDE \eqref{eq:fwd1}-\eqref{eq:bwd1}, \eqref{eq:nabla_X}-\eqref{eq:nabla_Y} and \eqref{eq:DX}-\eqref{eq:malliavin_Y} respectively. Then the following representation formulas hold: 
\begin{align}
\label{repformulaforDX}
D_u X_t &= \nabla X_t (\nabla X_u)^{-1} \sigma(u,X_u)\1_{\{u\leq t\}},& t,u\in[0,T],\ \ud\IP-a.s.\\
\nonumber
D_u Y_t &= \nabla Y_t (\nabla X_u)^{-1} \sigma(u,X_u)\1_{\{u\leq t\}},& t,u\in[0,T],\ \ud \IP-a.s.\\
\label{repformulaforZ}
Z_t &= \nabla Y_t (\nabla X_t)^{-1} \sigma(t,X_t),& t\in[0,T],\ \ud \IP\otimes \ud t-a.s.\\
\nonumber
D_u Z_t &= \nabla Z_t (\nabla X_u)^{-1} \sigma(t,X_u)\1_{\{u\leq t\}},& t,u\in[0,T],\ \ud \IP\otimes \ud t-a.s.
\end{align}
\end{theorem}
\begin{proof}
As in Theorem \ref{malliavindifftheo} we remark briefly that the properties of the forward component are well known and hence equality \eqref{repformulaforDX} holds, see \cite{nualart1995} or \cite{ImkellerDosReis}. Theorem \ref{malliavindifftheo} ensures that $(DX,DY,DZ)$ is the unique solution of the time delayed FBSDE \eqref{eq:DX}-\eqref{eq:malliavin_Y}. Throughout let $t \in [0,T]$ and $u \in [0,t]$. We define the processes
\[
U_{u,t}=\nabla Y_t (\nabla X_u)^{-1} \sigma(X_u) \1_{\{u\leq t\}}
\ \text{ and }\
V_{u,t}=\nabla Z_t (\nabla X_u)^{-1} \sigma(X_u) \1_{\{u\leq t\}},
\]
and for $s\in [0,T]$, we set $D_u X(s) = \int_{-T}^0 D_u X_{s+v} \ax( \ud v )$, 
\begin{align*}
U_u (s) &= \int_{-T}^0 U_{u,s+v} \ay(\ud v) = \int_{-T}^0 \nabla Y_{s+v} \big( \nabla X_u \big)^{-1} \sigma(u,X_u) \1_{\{u \leq s+v\}} \ay(\ud v),\\
V_u (s) &= \int_{-T}^0 V_{u,s+v} \az(\ud v) = \int_{-T}^0 \nabla Z_{s+v} \big( \nabla X_u \big)^{-1} \sigma(u,X_u) \1_{\{u \leq s+v\}} \az(\ud v),
\end{align*}
compare also with the notation in \eqref{eq:notation1}. Multiplying the BSDE \eqref{eq:nabla_Y} with $(\nabla X_u)^{-1} \sigma(u,X_u)$ and then using \eqref{repformulaforDX} we obtain for any $0\leq u\leq t\leq T$ $\ud \IP$-a.s. that
\begin{align*}
U_{u,t} &= \nabla g(X_T)D_u X_T - \int_t^T V_{u,s} \udws\\
& \qquad + \int_t^T 
\big\langle (\nabla f)\big(s, \Theta(s)\big)
, \big( D_{u}X(s), U_{u}(s), V_{u}(s)\big) \big\rangle \uds,
\end{align*}
where $\Theta$ is given by $\Theta(\cdot) = \big((X \cdot \ax)(\cdot),(Y\cdot \ay)(\cdot),(Z\cdot \az)(\cdot)\big)$ (compare with \eqref{eq:theta} from section \ref{section:diff}). Now, Theorem \ref{malliavindifftheo} states that the solution of BSDE \eqref{eq:malliavin_Y} is unique, hence  $(U,V)$ must coincide with $(DY,DZ)$. Another way to see this would be to use the a priori estimates of Proposition \ref{lemma:apriori} with \eqref{eq:malliavin_Y} and the above BSDE.

Formula  \eqref{repformulaforZ} follows easily from a combination of the representation formula for $D_u Y_t$ combined with $D_t Y_t = Z_t$, $\ud \IP\otimes \ud t$-a.s. (see Theorem \ref{malliavindifftheo}).
\end{proof}

\subsection*{Implications of the representation formula}

The representation formulas in the previous theorem allow for a deeper analysis of the control process $Z$ concerning its path properties. 
\begin{theorem}\label{continuitytheorem}
Let $p\geq 2$, assume that $|f(\cdot,0,0,0)|$ is uniformly bounded 
and that the conditions of Corollary \ref{cor:nabla_Y_exists} hold.
Then for $p \geq 2$, the mapping $t\mapsto Z_t$ is continuous $\ud \IP$-a.s.  If moreover we have $p>2$, then we also have \[\|Z\|_{\cS^q_0}<\infty \ \text{ for }q\in[2,p).\] In particular, for $p> 2$ we have for every $s,t \in [0,T]$ that $\IE\big[\,|Y_t-Y_s|^p\big]\leq C |t-s|^{p/2}$ and that $Y$ has continuous paths.
\end{theorem}
\begin{proof}
It is fairly easy to show that $\big(\nabla Y_t (\nabla X_t)^{-1} \sigma(t,X_t)\big)_{t\in[0,T]}$ is continuous. By assumption, $\sigma$ is a continuous function and it is well known that both processes $(\nabla X)^{-1}$ and $X$ have continuous paths. $\nabla Y$ is continuous because its dynamics is given as a sum of a stochastic integral of a predictable process against a Brownian motion (so a continuous martingale) and a Lebesgue integral with well behaved integrand. If two processes are versions of each other and one is continuous then they are in fact modifications of each other and hence $Z$ has continuous paths. 
Now since $Z$ has continuous paths, then the representation formula \eqref{repformulaforZ} does not only hold $\ud\IP\otimes\ud t$-almost surely but in fact holds for all $t\in[0,T]$ and $\IP$-almost all $\omega\in\Omega$. Using that $\nabla Y \in \cS^p_0$ for some $p>2$ (see Corollary \ref{cor:nabla_Y_exists} and Proposition \ref{prop:diff}), $(\nabla X)^{-1},\sigma(\cdot,X)\in \cS^q_0$ for any $r\geq 2$ and H\"older's inequality, we conclude that $Z\in\cS^q_0$ for every $q\in[2,p)$.

\vspace{0.3cm}
The property concerning the increments of $Y$ is easy to prove since $X,Y,Z\in \cS^p_0$ for some $p> 2$.  For $0\leq s\leq t\leq T$, we have (recall that $| f(\cdot,\Theta(\cdot)) | \leq |f(\cdot,\Theta(\cdot)) - f(\cdot,0,0,0)| + |f(\cdot,0,0,0)|$ and that $|f(\cdot,0,0,0)|$ is uniformly bounded)
\begin{align*}
Y_t-Y_s &= 0 +\int_s^t f\big(u, \Theta(u)\big)\ud u-\int_s^t Z_u\ud W_u,
\end{align*}
so using the assumptions and the Burkholder-Davis-Gundy inequality, we get for a generic constant $C$ which may vary from line to line and some $p> 2$
\begin{align*}
\IE\big[\,|Y_t-Y_s|^p\big] &\leq C\, \IE\Big[ \,\Big|\int_s^t f\big(u, \Theta(u)\big)\ud u\Big|^p+\Big|\int_s^t Z_u\ud W_u\Big|^p \Big]\\
&\leq C\, |t-s|^{p/2} \big( 1 + \|X\|_{\cS^p_0}^p+ \|Y\|_{\cS^p_0}^p + \|Z\|_{\cS^p_0}^p\big)+
\IE\big[\Big(\int_s^t |Z_u|^2\ud u\Big)^{p/2} \big]\\
&\leq C\, |t-s|^{p/2}. 
\end{align*}
This in particular yields the applicability of Kolmogorov's continuity criterion to $Y$.
\end{proof}

\subsubsection*{The $L^2$-regularity result}

We finish this section with the $L^2$-regularity result for the control component $Z$ of the solution of the time delayed FBSDE. Let $\pi$ be a partition of the time interval $[0,T]$ with $N$ points and mesh size $|\pi|$. We define a set of random variables via
\begin{align*}
\bar{Z}^\pi_\ti&=\frac1{\tip-\ti}\E\Big[\int_\ti^\tip
Z_s\uds\big|\cF_\ti\Big], \textrm{ for all partition
points } t_i,\ 0\le i\le N-1.
\end{align*}
The best square integrable $\cF_{t_i}$-measurable approximation of $\frac{1}{t_{i+1}-t_i}\int_\ti^\tip Z_s\uds$ is given by $\bar{Z}^\pi_\ti$, i.e.
\begin{align}\label{eq:leastsquare}
\IE \Big[\, \big| \frac{1}{t_{i+1}-t_i}\int_\ti^\tip Z_s\uds -  \bar{Z}^\pi_\ti \big|^2 \Big] &= \inf_{V \in L^2(\cF_{t_i})} \IE \Big[\, \big| \frac{1}{t_{i+1}-t_i}\int_\ti^\tip Z_s\uds -  V \big|^2 \Big].
\end{align}
We associate the process $(\bar{Z}^\pi_t)_{t\in[0,T]}$ to $\{\bar{Z}^\pi_\ti\}_{i=0,\cdots,N-1}$ via $\bar{Z}^\pi_t = \bar{Z}^\pi_\ti$ for $t\in[\ti, \tip),\, 0\le i\le N-1$. Similarly, for the set of random variables 
$\{Z_\ti:\ti\in\pi\}$, we associate the process $(Z^\pi_t)_{t\in[0,T]}$ via $Z^\pi_t = Z^\pi_\ti$ for $t\in[\ti, \tip),\, 0\le i\le N-1$. The definition of the conditional expectation implies that for every $i=0,\ldots,N-1$, we have
$$ \IE[\,|Z^\pi_{\ti}|^2] -2\, \IE[\,Z^\pi_{\ti}\, \bar{Z}^\pi_{\ti}\,] \geq -\E[\,|\bar{Z}^\pi_{\ti}|^2],$$
from which it follows that $\bar{Z}^\pi$ is the best $\cH^2$-approximation of $Z$, leading to
\[
\| Z-\bar Z^\pi  \|_{\cH^2}\leq \|Z-Z^\pi \|_{\cH^2}\to 0,\ \textrm{ as }\ |\pi|\to 0.
\]
Using Theorem \ref{continuitytheorem} we are able to determine explicitly the rate of convergence of the above limit. The following result extends Theorem 5.6 from \cite{ImkellerDosReis} to the setting of FBSDE with time delayed generators. 

\begin{theorem}[$L^2$-regularity]
\label{theo:l^21regularity}
Assume that the conditions of Theorem \ref{continuitytheorem} hold for some $p>2$ and assume further that $\sigma$ is $\frac12$-H\"older continuous function in its time variable. Then
\begin{align*}
\max_{0\leq i\leq N-1}\Big\{
\sup_{\ti\leq t\leq \tip}
\IE\big[\, |Y_t -Y_\ti|^2\ \big]\, \Big\}+
\sum_{i=0}^{N-1} \IE\Big[ \int_{t_i}^{t_{i+1}}|Z_s-\bar{Z}^\pi_{t_i}|^2\uds \Big]\leq C |\pi|.
\end{align*}
\end{theorem}

\begin{proof}
The result concerning the $Y$ component follows immediately from Theorem \ref{continuitytheorem}. As for the result for $Z$, let us remark that since $\bar{Z}^\pi$ is the best $\cH^2$-approximation of $Z$ over $\pi$ in the sense of \eqref{eq:leastsquare}, it follows that 
\[
\sum_{i=0}^{N-1} \IE\Big[ \int_{t_i}^{t_{i+1}}|Z_s-\bar{Z}^\pi_{t_i}|^2\uds \Big]
\leq
\sum_{i=0}^{N-1} \IE\Big[ \int_{t_i}^{t_{i+1}}|Z_s-Z_{t_i}|^2\uds \Big]
=
\sum_{i=0}^{N-1} \int_{t_i}^{t_{i+1}}\IE\big[\, |Z_s-Z_{t_i}|^2 \big] \uds,
\]
where the last equality follows from the use of Fubini's theorem to switch the integration order (recall that $Z\in\cS^p_0$ for some $p>2$).
Theorem \ref{continuitytheorem} allows to use \eqref{repformulaforZ} to rewrite the difference inside the expectation. We have $Z_s-Z_\ti= I_1+I_2+I_3$
with $I_1=[\nabla Y_s-\nabla Y_\ti] (\nabla X_\ti)^{-1}\sigma(\ti,X_\ti)$,
$I_2=\nabla Y_s[(\nabla X_s)^{-1}-(\nabla X_\ti)^{-1}]\sigma(\ti,X_\ti)$, $I_3=\nabla Y_s(\nabla X_s)^{-1}[\sigma(s,X_s)-\sigma(\ti,X_\ti)]$ and $s\in[\ti,\tip]$.

\smallskip

From the proof of part (ii) of Theorem 5.8 in \cite{pathregcorrection2010} one obtains that
\[
\sum_{i=0}^{N-1} \IE\Big[ \int_{t_i}^{t_{i+1}} |I_2|^2\uds +\int_{t_i}^{t_{i+1}} |I_3|^2\uds \Big] \leq C|\pi|.
\]
The calculations that lead to the above result are quite easy to carry out. They rely on known estimates for SDEs found for instance in Theorem 2.3 and 2.4 of \cite{ImkellerDosReis} combined with the fact that $\nabla Y\in\cS^p$ for some $p>2$.

\smallskip

To handle the term $I_1$ one needs to proceed with more care. Let us start with a simple trick:
\begin{align}
\label{trickwithconditionalexpectation}
\IE\Big[\,|(\nabla Y_s-\nabla Y_\ti) (\nabla X_\ti)^{-1}\sigma(\ti,X_\ti)|^2\Big]
=
\IE\Big[\,\IE\big[\,|\nabla Y_s-\nabla Y_\ti|^2\big|\cF_\ti\big] |(\nabla X_\ti)^{-1}\sigma(\ti,X_\ti)|^2\Big].
\end{align}
Writing the BSDE for the difference $\nabla Y_s-\nabla Y_\ti$ for $s\in[\ti,\tip]$ we get for a generic constant $C>0$ that 
\begin{align*}
\IE\Big[\,|\nabla Y_s-\nabla Y_\ti|^2\Big|\cF_\ti\Big] &
\leq C\,
\IE\Big[\,|\int_\ti^s \big\langle (\nabla f)\big(r,\Theta(r)\big), (\nabla \Theta)(r)\big\rangle\ud r|^2+\big|\int_\ti^s \nabla Z_r\ud W_r\big|^2\Big|\cF_\ti\Big]
\\
& \leq C\,
\IE\Big[\,|\pi|\int_\ti^\tip \big |(\nabla \Theta)(r)|^2\ud r + \int_\ti^\tip |\nabla Z_r|^2\ud r\Big|\cF_\ti\Big],
\end{align*}
where we used the uniform boundedness of the derivatives of $f$, Jensen's inequality, It\^o's isometry and proceeded to maximize over the time interval $[\ti,\tip]$. Combining the last line with \eqref{trickwithconditionalexpectation} and using the tower property, we obtain
\begin{align*}
& \sum_{i=0}^{N-1} \int_\ti^\tip \IE\Big[\,\IE\Big[\,|\nabla Y_s-\nabla Y_\ti|^2\Big|\cF_\ti\Big] |(\nabla X_\ti)^{-1}\sigma(\ti,X_\ti)|^2\Big]\uds\\
&\qquad \leq C \sum_{i=0}^{N-1} |\pi| 
\IE\Big[\Big( |\pi|\int_\ti^\tip \big |(\nabla \Theta)(r)|^2\ud r + \int_\ti^\tip |\nabla Z_r|^2\ud r \Big) |(\nabla X_\ti)^{-1}\sigma(\ti,X_\ti)|^2\Big]\\
&\qquad \leq  |\pi| 
\IE\Big[
\sup_{0\leq t\leq T}|(\nabla X_t)^{-1}\sigma(t,X_t)|^2\,
\sum_{i=0}^{N-1}
\Big( |\pi|\int_\ti^\tip \big |(\nabla \Theta)(r)|^2\ud r + \int_\ti^\tip |\nabla Z_r|^2\ud r \Big) \Big]\\
&\qquad =  |\pi| \IE\Big[
\sup_{0\leq t\leq T}|(\nabla X_t)^{-1}\sigma(t,X_t)|^2\,
\Big( |\pi|\int_0^T \big |(\nabla \Theta)(r)|^2\ud r + \int_0^T |\nabla Z_r|^2\ud r \Big) \Big]\\
& \qquad \leq C |\pi|,
\end{align*}
where in the last line we used the fact that $\nabla X, (\nabla X)^{-1},X \in \cS^q_0$ for every $q\geq 2$ and that $\nabla Y, \nabla Z \in\cH^p_0$ for some $p> 2$ (in combination with H\"older's inequality) to conclude the finiteness of the expectation. 
Combining this estimate with the ones for $I_2$ and $I_3$ finishes the proof.
\end{proof}

\subsection*{Towards a time discretization of delay FBSDE}

Having established a path regularity result for FBSDE with time-delayed generators one can now start discussing a working numerical scheme. Given the nature of this class of BSDE, a time discretization would naturally require some decoupling technique to handle the backward-in-time feature of the equation and the backward-in-time feature of the delay. 

Applying the backward time discretization from \cite{04BT} to \eqref{eq:fwd1}-\eqref{eq:bwd1}, we obtain for a partition $\pi: 0=t_0 < t_1 < \ldots < t_N = T$ with step size $\Delta_i=t_{i+1}-t_i$
\begin{align*}
Y^\pi_{t_N} &= g( X^\pi_{t_N}),
\\
Z^\pi_{t_i} &= \IE \Big[ \frac{W_{t_{i+1}} - W_{t_{i} } }{ \Delta_i } Y^\pi_{t_{i+1}} | \cF_{t_i} \Big],
\qquad 
Y^\pi_{t_i} = \IE \Big[ Y^\pi_{t_{i+1}} | \cF_{t_i} \Big] + \Delta_i\, f( t_i, \Theta^\pi_{t_i} ),
\\
\text{where }\quad & \Theta^{\pi}_{t_i} = 
\Big( \sum_{j=0}^{i} X^{\pi}_{t_j} \ax\big([t_j,t_{j+1})\big) , \sum_{j=0}^{i} Y^{\pi}_{t_j} \ay\big([t_j,t_{j+1})\big) , \sum_{j=0}^{i} Z^{\pi}_{t_j} \az\big([t_j,t_{j+1})\big)\Big).
\end{align*}
This backward scheme cannot be implemented because in the computation of each $Y^\pi_{t_i}$ running backward from  $i=N-1$ to $i=0$, we must evaluate $\Theta^\pi(t_i)$ which depends on all $Y^\pi_{t_j}$, $Z^\pi_{t_j}$ running in forward direction $j=0,\ldots,i$.

\smallskip

However, \cite{BenderDenk} propose for standard Lipschitz BSDEs a time discretization which mimics the Picard iteration technique for proving existence and uniqueness of BSDEs. Due to the fact that in each iteration step, one solves an explicit BSDE, the scheme from \cite{BenderDenk} runs \emph{forward} in time. The price to pay is to control apart from the error contribution of the time discretization the additional error arising from the Picard iterates (see Theorem 2 in \cite{BenderDenk}).  This idea adapts to equations \eqref{eq:fwd1}-\eqref{eq:bwd1} by exploiting the fact that the solution $(Y,Z)$ is obtained as a limit of $(Y^p,Z^p)$ as $p$ goes infinity. Setting up $(Y^0,Z^0) = (0,0)$ and then for $p\in \IN_0$ we have
\begin{align*}
Y^{p+1}_t & 
= g(X_T) + \int_t^T f\big( s,\Theta^{p}(s) \big) \uds - \int_t^T Z^{p+1}_s \udws,\quad t\in[0,T]
\\
\text{where }\quad & \Theta^{p}(t)=
\Big(\int_{-T}^0 X_{t+v} \ax(\ud v), \int_{-T}^0 Y^p_{t+v} \az(\ud v), \int_{-T}^0 Z^p_{t+v} \az(\ud v) \Big).
\end{align*}
The discretization hereof is for $p \in \IN_0$, initiated by setting  $(Y^{\pi,0},Z^{\pi,0}) = (0,0)$, then iteratively for $p\geq 1$ and  $0\leq i\leq N-1$
\begin{align*} 
Y^{\pi,p+1}_{t_i} &= \E \Big[\ g\big( X^{\pi}_{t_N} \big) + \sum_{j=i}^{N-1} 
f(t_j,\Theta^{\pi,p}_\tj ) \Delta_j \ \big| \cF_{t_i}\Big],
\\
Z^{\pi,p+1}_{t_i} &= \E \Big[\ \frac{ W_{t_{i+1}} - W_{t_i} }{ \Delta_i} \Big( g(X^\pi_{t_N})+\sum_{j=i+1}^{N-1} f(t_j,\Theta^{\pi,p}_\tj \big)\Delta_j  \Big) \big| \cF_{t_i}\Big],
\\
\text{where }\quad
& \Theta^{\pi,p}_\ti = 
\Big( \sum_{j=0}^{i} X^{\pi}_{t_j} \ax\big([t_j,t_{j+1})\big) , \sum_{j=0}^{i} Y^{\pi,p}_{t_j} \ay\big([t_j,t_{j+1})\big) , \sum_{j=0}^{i} Z^{\pi,p}_{t_j} \az\big([t_j,t_{j+1})\big)\Big).
\end{align*} 
The proof of convergence for this time discretization scheme is left for future research.

\section*{Acknowledgments}

The authors are grateful to the anonymous referee for suggestions and comments which have greatly improved the readability of the paper.

\bibliographystyle{abbrvnat}

\begin{thebibliography}{15}
\providecommand{\natexlab}[1]{#1}
\providecommand{\url}[1]{\texttt{#1}}
\expandafter\ifx\csname urlstyle\endcsname\relax
  \providecommand{\doi}[1]{doi: #1}\else
  \providecommand{\doi}{doi: \begingroup \urlstyle{rm}\Url}\fi

\bibitem[Ankirchner et~al.(2007)Ankirchner, Imkeller, and
  Dos~Reis]{AnkirchnerImkellerDosReis}
S.~Ankirchner, P.~Imkeller, and G.~Dos~Reis.
\newblock Classical and variational differentiability of {BSDE}s with quadratic
  growth.
\newblock \emph{Electron. J. Probab.}, 12:\penalty0 1418--1453, 2007.

\bibitem[Bender and Denk(2008)]{BenderDenk}
C.~Bender and R.~Denk.
\newblock A forward scheme for backward {SDE}s.
\newblock \emph{Stochastic Process. Appl.}, 117\penalty0 (12):\penalty0 1793--1812, 2007.

\bibitem[Bender and Zhang(2008)]{BenderZhang2008}
C.~Bender and J.~Zhang.
\newblock Time discretization and {M}arkovian iteration for coupled {FBSDE}s.
\newblock \emph{Ann. Appl. Probab.}, 18\penalty0 (1):\penalty0 143--177, 2008.
\newblock ISSN 1050-5164.

\bibitem[Bouchard and Touzi(2004)]{04BT}
B.~Bouchard and N.~Touzi.
\newblock Discrete-time approximation and {M}onte-{C}arlo simulation of
  backward stochastic differential equations.
\newblock \emph{Stochastic Process. Appl.}, 111\penalty0 (2):\penalty0
  175--206, 2004.

\bibitem[Delong(2010)]{Delong2010}
{\L}.~Delong.
\newblock Applications of time-delayed backward stochastic differential
  equations to pricing, hedging and management of insurance and financial
  risks.
\newblock Preprint - arXiv:1005.4417, 2010.

\bibitem[Delong and Imkeller(2010{\natexlab{a}})]{DelongImkeller}
L.~Delong and P.~Imkeller.
\newblock Backward stochastic differential equations with time delayed
  generators - results and counterexamples.
\newblock \emph{Ann. Appl. Probab.}, 20\penalty0 (4):\penalty0 1512--1536,
  2010{\natexlab{a}}.

\bibitem[Delong and Imkeller(2010{\natexlab{b}})]{DelongImkeller2}
L.~Delong and P.~Imkeller.
\newblock On {M}alliavin's differentiability of {BSDE} with time delayed
  generators driven by {B}rownian motions and {P}oisson random measures.
\newblock \emph{Stochastic Process. Appl.}, 120\penalty0 (9):\penalty0
  1748--1775, August 2010{\natexlab{b}}.

\bibitem[El~Karoui et~al.(1997)El~Karoui, Peng, and Quenez]{97KPQ}
N.~El~Karoui, S.~Peng, and M.~C. Quenez.
\newblock Backward stochastic differential equations in finance.
\newblock \emph{Math. Finance}, 7\penalty0 (1):\penalty0 1--71, 1997.

\bibitem[Imkeller and Dos~Reis(2010{\natexlab{a}})]{ImkellerDosReis}
P.~Imkeller and G.~Dos~Reis.
\newblock Path regularity and explicit convergence rate for {BSDE} with
  truncated quadratic growth.
\newblock \emph{Stochastic Processes Appl.}, 120\penalty0 (3):\penalty0
  348--379, 2010{\natexlab{a}}.

\bibitem[Imkeller and Dos~Reis(2010{\natexlab{b}})]{pathregcorrection2010}
P.~Imkeller and G.~Dos~Reis.
\newblock Corrigendum to ``path regularity and explicit convergence rate for
  {BSDE} with truncated quadratic growth'' [stochastic process. appl. 120
  (2010) 348-379].
\newblock \emph{Stochastic Process. Appl.}, 120\penalty0 (11):\penalty0 2286 --
  2288, November 2010{\natexlab{b}}.

\bibitem[Karatzas and Shreve(1995)]{KaratzasShreve}
I.~Karatzas and S.~Shreve.
\newblock \emph{Brownian motion and Stochastic calculus}.
\newblock Vol. 113 of \emph{Graduate Texts in Mathematics}, Springer-Verlag (New-York), 
1991.

\bibitem[Khoshnevisan(2002)]{Khosh}
D.~Khoshnevisan.
\newblock \emph{{Multiparameter processes. An introduction to random fields.}}
\newblock {Springer Monographs in Mathematics. New York, NY: Springer. xix, 584
  p.}, 2002.

\bibitem[Nualart(1995)]{nualart1995}
D.~Nualart.
\newblock \emph{The {M}alliavin calculus and related topics}.
\newblock Probability and its Applications (New York). Springer-Verlag, New
  York, 1995.

\bibitem[Pardoux and Peng(1990)]{PardouxPeng90}
E.~Pardoux and S.~Peng.
\newblock {Adapted solution of a backward stochastic differential equation.}
\newblock \emph{Syst. Control Lett.}, 14\penalty0 (1):\penalty0 55--61, 1990.

\bibitem[Protter(2005)]{Protter}
P.~E. Protter.
\newblock \emph{Stochastic integration and differential equations}.
\newblock Applications of Mathematics (New York). Springer-Verlag, 2nd edition,
  2005.
\newblock Version 2.1.

\bibitem[Shiryaev(1995)]{Shiryaev}
A.~N. Shiryaev.
\newblock \emph{Probability. Transl. from the Russian by R. P. Boas. 2nd ed.}
\newblock Graduate Texts in Mathematics. 95. New York, NY: Springer-Verlag.,
  1995.

\bibitem[Wang et~al.(2007)Wang, Ran, and Chen]{WangRanChen}
J.~Wang, Q.~Ran, and Q.~Chen.
\newblock {$L^p$} solutions of {BSDE}s with stochastic {L}ipschitz condition.
\newblock \emph{J. Appl. Math. Stoch. Anal.}, pages Art. ID 78196, 14, 2007.

\bibitem[Yong and Zhou(1999)]{YongZhou}
J.~Yong and X.~Y. Zhou.
\newblock \emph{{Stochastic controls. Hamiltonian systems and HJB equations.}}
\newblock {Applications of Mathematics. 43. New York, NY: Springer. xx, 438
  p.}, 1999.

\end{thebibliography}

\end{document}